\documentclass{amsart}
\usepackage{tikz}
\title[Schatten class Hankel operators on doubling Fock spaces]{Schatten class Hankel operators on doubling Fock spaces and the Berger-Coburn phenomenon}
\author[G. Asghari]{Ghazaleh Asghari}
\address{Department of Mathematics and Statistics, University of Reading, England}
\email{g.asgharikhonakdari@pgr.reading.ac.uk}

\author[Z. Hu]{Zhangjian Hu}
\address{Department of Mathematics, Huzhou University, Huzhou, China}
\email{huzj@zjhu.edu.cn}

\author[J. A. Virtanen]{Jani A. Virtanen}
\address{Department of Mathematics and Statistics, University of Reading, England}
\email{j.a.virtanen@reading.ac.uk}

\thanks{G. Asghari was supported by EPSRC grant EP/W524128/1. Z. Hu was supported in part by the National Natural Science Foundation of China (12071130, 12171150). J. Virtanen was supported in part by EPSRC grant EP/X024555/1.}

\usepackage{amsmath}
\usepackage{amssymb}
\usepackage{amsfonts}
\usepackage{amsthm}
\usepackage{mathtools}
\usepackage{enumerate}
\usepackage[toc,page]{appendix}
\usepackage[utf8]{inputenc}
\usepackage[style=numeric]{biblatex}
\usepackage{yfonts}
\usepackage{setspace}

\usepackage{xcolor}
\usepackage{comment}
\usepackage{graphicx}

\theoremstyle{definition}

\theoremstyle{remark}
\newtheorem*{rem}{Remark}

\theoremstyle{plain}
\newtheorem{thm}{Theorem}[section]
\newtheorem{prop}[thm]{Proposition}
\newtheorem{lem}[thm]{Lemma}
\newtheorem{cor}[thm]{Corollary}
\newtheorem{OP}[thm]{Open Problem}

\theoremstyle{definition}
\newtheorem{Rem}[thm]{Remark}

\usepackage{scalerel,stackengine}
\stackMath
\newcommand\reallywidehat[1]{
\savestack{\tmpbox}{\stretchto{
  \scaleto{
    \scalerel*[\widthof{\ensuremath{#1}}]{\kern-.6pt\bigwedge\kern-.6pt}
    {\rule[-\textheight/2]{1ex}{\textheight}}
  }{\textheight} 
}{0.5ex}}
\stackon[1pt]{#1}{\tmpbox}
}
\numberwithin{equation}{section}
 
\theoremstyle{plain}

\theoremstyle{definition}

\newtheorem{?}[Th]{Problem}

\usepackage[utf8]{inputenc}
\usepackage{amsmath}
\usepackage{amssymb}
\usepackage{amsfonts}
\usepackage{amsthm}
\usepackage{mathtools}
\usepackage{enumerate}
\usepackage[toc,page]{appendix}

\DeclareMathOperator{\IDA}{IDA}
\DeclareMathOperator{\IMO}{IMO}

\begin{document}
\thispagestyle{empty}
  
\begin{abstract}
Using the notion of integral distance to analytic functions, we give a characterization of Schatten class Hankel operators acting on doubling Fock spaces on the complex plane and use it to show that for $f\in L^{\infty}$, if $H_{f}$ is Hilbert-Schmidt, then so is $H_{\bar{f}}$. This property is known as the Berger-Coburn phenomenon. When $0<p\le 1$, we show that the Berger-Coburn phenomenon fails for a large class of doubling Fock spaces. Along the way, we illustrate our results for the canonical weights $|z|^m$ when $m>0$.
\end{abstract}

\maketitle

\section{Introduction and Main Results}
Let $dA=\frac{1}{2i}dz\wedge d\Bar{z}$ be the Lebesgue measure on $\mathbb{C}$, and $\phi$ be a subharmonic function. For $0<p<\infty$, $L^{p}_{\phi}=L^{p}(\mathbb{C}, e^{-p\phi}dA)$ is the space of all measurable functions on $\mathbb{C}$ such that
\begin{equation} \|f\|^{p}_{p,\phi}=\int_{\mathbb{C}}|f(z)|^{p}e^{-p\phi(z)}dA(z)<\infty,
\end{equation}
and $L^{\infty}_{\phi}$ is the space of measurable functions $f$ such that
\begin{equation}
\|f\|_{\infty,\phi}=\textrm{ess}\sup_{z\in\mathbb{C}}|f(z)|e^{-\phi(z)}~<\infty.
\end{equation}
Moreover, we write $L^{p}(\Omega)$ for the space $L^{p}(\Omega,dA)$ where $\Omega\subset\mathbb{C}$, and we abbreviate $L^{p}(\mathbb{C},dA)$ as $L^{p}$. A positive Borel measure $\mu$ on $\mathbb{C}$ is called doubling if there exists some constant $C>1$ such that 
\begin{equation}
    \mu(D(z,2r))\leq C\mu(D(z,r))
\end{equation}
for all $z\in \mathbb{C}$ and $r>0$, where $D(z,r)$ is the open disk in $\mathbb{C}$ with center $z$ and radius $r$. The smallest $C>1$ is called the doubling constant for $\mu$. Hence, for each $z\in\mathbb{C}$, $\lim_{r\to\infty}\mu(D(z,r))=\infty$. It is well known that $\mu$ has no point mass, i.e.,
\begin{equation}
    \mu(\partial D(z,r))=\mu(\{z\})=0\quad \textrm{for every $z\in\mathbb{C}$ and $r>0$},
\end{equation}
and is nonzero and locally finite. That is,
\begin{equation}
    0<\mu(D(z,r))<\infty \quad\textrm{for every $z\in\mathbb{C}$ and $r>0$}.
\end{equation}
Note that since for each $z\in \mathbb{C}$, $\lim_{r\to\infty}\mu(D(z,r))=\infty$,  the function $r\mapsto\mu(D(z,r))$ is an increasing homeomorphism from $(0,\infty)$ to itself. Therefore, for every $z\in\mathbb{C}$, there is a unique positive radius $\rho(z)$ such that $\mu(z,\rho(z))=1$. For more information on doubling measures see \cite{stein1993harmonic}. Denote by $H(\mathbb{C})$ the space of holomorphic functions on $\mathbb{C}$. Then the doubling Fock space $F^{p}_{\phi}$ is defined by
\begin{equation}  F^{p}_{\phi}=L^{p}_{\phi}\cap H(\mathbb{C})
\end{equation}
where $\phi$ is a subharmonic function, not identically zero on $\mathbb{C}$, and $d\mu=\triangle\phi \, dA$ is a doubling measure. As shown in 
\cite{marco2003interpolating},
$\rho^{-2}$ is a regularization of $\Delta\phi$. Indeed, Theorem 14 in \cite{marco2003interpolating} states that when $\phi$ is subharmonic and $\Delta\phi\, dA$ is a doubling measure, there exists a subharmonic function $\psi\in \mathcal{C}^{\infty}(\mathbb{C})$ and $C>0$ such that $|\psi-\phi|\leq C$, $\Delta\psi\, dA$ a doubling measure, and $\Delta\psi \simeq\rho^{-2}_{\psi}\simeq\rho^{-2}_{\phi}$. The comparability relation $\simeq$ is explained at the beginning of Section 2. Since the spaces of functions and sequences that we consider do not change if $\phi$ is replaced by $\psi$, we will assume that $\phi\in \mathcal{C}^{\infty}(\mathbb{C})$ and $\Delta\phi dA\simeq dA/\rho^{2}$ is a doubling measure. Hence, up to normalization by a constant, we can consider $\rho^{-2}(z)dz\otimes d\Bar{z}$ to be the metric tensor describing the underlying geometry of our space.

It is well known that $(F^{p}_{\phi},\|\cdot\|_{p,\phi})$ is a Banach space for $1\leq p\leq\infty$ and a quasi-Banach space for $0<p<1$. Let $K_{z}=K(\cdot,z)$ be the reproducing kernel of $F^{2}_{\phi}$. Then the orthogonal projection $P:L^{2}_{\phi}\to F^{2}_{\phi}$ is given by
\begin{equation}
Pf(z)=\int_{\mathbb{C}}f(w)\overline{K_{z}(w)}e^{-2\phi(w)}dA(w).
\end{equation}
Then as shown in \cite{oliver2016toeplitz}, for any $1\leq p\leq\infty$, $P$ is a bounded linear operator from $L^{p}_{\phi}$ to $F^{p}_{\phi}$, and for any $f\in F^{p}_{\phi}$, $f=Pf$. Let $\Gamma=\operatorname{span}\{K_{z}:z\in \mathbb{C}\}$, and consider the class of symbols
\begin{equation*}
    \mathcal{S}=\{ f \textrm{ measurable}:fg\in L^{2}_{\phi}\textrm{ for } g\in \Gamma\}.
\end{equation*}
Note that $L^{\infty}\subset \mathcal{S}$. Given $f\in \mathcal{S}$, define the Toeplitz operator $T_{f}$ and the Hankel operator $H_{f}$ on $F^{p}_{\phi}$ by
\begin{equation}
    T_{f}g=P(fg),\quad H_{f}g=(I-P)(fg)=fg-P(fg).
\end{equation}
The doubling Fock spaces as well as some pointwise estimates of the Bergman kernel have been studied in seminal papers of Christ \cite{MR1120680}, and Marco, Massaneda and Ortega-Ceda~\cite{marco2003interpolating,marzo2009pointwise}. Oliver and Pascuas~\cite{oliver2016toeplitz} studied the characterization of boundedness, compactness and the Schatten class membership of Toeplitz operators on doubling Fock spaces.  In \cite{Jani2021ida}, Hu and Virtanen introduced a new space $\operatorname{IDA}$ of locally integrable functions whose integral distance to holomorphic functions is finite and used it to characterize boundedness and compactness of Hankel operators on weighted Fock spaces. Using the same notion, in \cite{Janihu2022schatten} they characterized Schatten class Hankel operators acting on weighted Fock spaces $F^{2}_{\Phi}$, where $m \leq \triangle \Phi\leq M$ for some $m,M>0$. Recently, their characterizations of bounded and compact Hankel operators was extended to the setting of doubling Fock spaces in \cite{lv2023hankel}.

In the present work, we use a generalized version of $\operatorname{IDA}$ to study the Schatten class membership of Hankel operators on doubling Fock spaces. Of particular interest is the result of Berger and Coburn~\cite{berger1987toeplitzBCP} which says that, for $f\in L^{\infty}$, if $H_{f}$ is a compact operator acting on the classical Fock space $F^{2}$, then so is $H_{\Bar{f}}$. We refer to this property as the Berger-Coburn phenomenon and note that an analogous statement fails both in the Hardy and Bergman spaces (see, e.g., \cite{HaggerJani}). More recently, Berger and Coburn's result has been extended to Fock spaces with standard weights by Hagger and Virtanen \cite{HaggerJani} (using limit operator techniques as opposed to $C^*$-algebra techniques and Hilbert space methods) and to generalized Fock spaces $F^{p}_\Phi$ by Hu and Virtanen \cite{Jani2021ida}. Our approach is similar to that of \cite{Jani2021ida} except that we need to deal with more complicated geometry induced by the function $\rho$ arising in the study of doubling Fock spaces.

It is natural to ask whether the Berger-Coburn phenomenon also holds for Schatten class Hankel operators. Indeed, Bauer \cite{bauer2004hilbert} was the first to show that this property holds for Hilbert-Schmidt Hankel operators on $F^{2}$. Recently, Hu and Virtanen in \cite{Janihu2022schatten} proved that when $1<p<\infty$, $H_{f}$ acting on $F^{2}_{\Phi}$ is in the Schatten class $S_{p}$ if and only if $H_{\bar{f}}$ is in $S_{p}$. This was followed by the work of Xia \cite{xia2023berger}, in which he showed also that if $f(z)=1/z$ for $|z|>1$ and $f=0$ elsewhere, then $H_f$ acting on the classical Fock space $F^2$ is in the trace class while $H_{\bar f}$ is not. In his work, Xia employed a rather long and involved calculations using the standard basis vectors $e_k(z) = z^{k}/\sqrt{k!}$ and the reproducing kernel $K(z,w) = e^{z\Bar{w}}$. Observe that for non-standard weighted Fock spaces, there are no explicit formulas for the basis vectors or the reproducing kernel. To overcome this, Hu and Virtanen \cite{hu2023berger} used their characterizations of Schatten class Hankel operators to verify that Xia's example shows that the Berger-Coburn phenomenon fails for $S_p(F^2_\varphi, L^2_\varphi)$ when $0<m<\Delta \varphi < M$ and $0<p\leq 1$. Here, we use an analogous approach on doubling Fock spaces to prove the existence of the Berger-Coburn phenomenon for Hilbert-Schmidt Hankel operators. When $0<p\le 1$, we show that the Berger-Coburn phenomenon fails for some doubling Fock spaces---the larger the value of $p$, the fewer Fock spaces we can cover.  

To state our main results, following \cite{Jani2021ida,luecking1992characterizations} with a modification according to the doubling property of the measure under consideration, we define 
    \begin{equation}
         G_{q,r}(f)(z)=\inf \left\{\big( \frac{1}{|D^{r}(z)|}\int_{D^{r}(z)}|f-h|^{q} dA\big)^{1/q} : h\in H(D^{r}(z))\right\},
    \end{equation}
for $f\in L^{q}_{loc}$, $q\geq 1$ and $r>0$. Here $|D^{r}(z)|$ is the Lebesgue measure of $D^{r}(z):=D(z,r\rho(z))$. Now, for $0<p\leq\infty$, $1\leq q\leq\infty$, and $\alpha\in \mathbb{R}$, the space $\operatorname{IDA}^{p,q,\alpha}_{r}$ consists of all $f\in L^{q}_{loc}$ such that $\|f\|_{\operatorname{IDA}^{p,q,\alpha}_{r}}=\|\rho^{\alpha}G_{q,r}(f)\|_{L^{p}}< \infty$. Besides, for $f\in L^{1}_{loc}$, define $\hat{f}_{r}(z):=|D^{r}(z)|^{-1}\int_{D^{r}(z)}fdA$.

\begin{thm}[IDA decomposition]\label{Thm1.1}
Let $\phi\in \mathcal{C}^{\infty}(\mathbb{C})$ be subharmonic such that $d\mu=\Delta\phi dA$ is a doubling measure. Suppose that $1\leq q\leq \infty$, $0<p<\infty$, $\alpha\in\mathbb{R}$, and $f\in L^{q}_{loc}$. Then for $f\in \operatorname{IDA}^{p,q,\alpha}_{r}$, $f=f_{1}+f_{2}$ where $f_{1}\in \mathcal{C}^{2}(\mathbb{C})$ and
    \begin{equation}\label{D1}
        \rho^{1+\alpha}|\Bar{\partial} f_{1}|+\rho^{1+\alpha}(\widehat{|\Bar{\partial} f_{1}|^{q}}_{r})^{1/q}+ \rho^{\alpha}(\widehat{| f_{2}|^{q}}_{r})^{1/q} \in L^{p},
    \end{equation}
    for some (equivalent any) $r>0$, and 
    \begin{equation}\label{D2}
 \|f\|_{\operatorname{IDA}^{p,q,\alpha}_{r}} \simeq \inf\left\{
        \|\rho^{1+\alpha}(\widehat{|\Bar{\partial} f_{1}|^{q}}_{r})^{1/q}\|_{L^{p}}+ \|\rho^{\alpha}(\widehat{| f_{2}|^{q}}_{r})^{1/q}\|_{L^{p}}
       \right \},
    \end{equation}
    where the infimum is taken over all possible decompositions $f=f_{1}+f_{2}$, with $f_{1}$ and $f_{2}$ satisfying the
conditions in (\ref{D13}).
\end{thm}

Theorem \ref{Thm1.1} was stated in \cite{lv2023hankel} without proof. We believe that the proof is rather technical and not trivial at all. It appears that this theorem should be a natural extension of Theorem 3.8 in \cite{Jani2021ida}. However, bounding a solution to the $\Bar{\partial}$-equation in the doubling Fock space is problematic. 
\begin{thm}[Schatten class membership of Hankel operators]\label{Thm1.2}
    Let $0<p\leq \infty$, and $\phi\in \mathcal{C}^{\infty}(\mathbb{C})$ be subharmonic such that $d\mu:=\Delta\phi dA$ is a doubling measure. Then for $f\in\mathcal{S}$, the following are equivalent:
    \begin{itemize}
        \item[(1)] $H_{f}:F^{2}_{\phi}\to L^{2}_{\phi}$ is in $S_{p}$,
        \item[(2)] $f\in\operatorname{IDA}^{p,2,-2/p}_{r}$, for some (equivalent any) $r>0$. 
    \end{itemize}
    Moreover,
    \begin{equation}\label{Schatten norm}
        \|H_{f}\|_{S_{p}}\simeq \|f\|_{\operatorname{IDA}^{p,2,-2/p}_{r}}.
    \end{equation}
\end{thm}
\begin{rem}
    Assuming smoothness of $\rho^{-2}$, the condition for the $S_{p}$ membership of the Hankel operator on the doubling Fock space is equivalent to the condition that $G_{2,r}(f)$ belongs to the space of $L^{p}$ functions on $\mathbb{C}$ with the conformal metric $\rho^{-2}dz\otimes d\bar{z}$.
\end{rem}

To characterize the simultaneous membership of $H_f$ and $H_{\bar f}$ in $S_p$, we need to define the space of integral mean oscillation. First, for $f\in L^{2}_{loc}$ and $r>0$, the mean oscillation of $f$ is defined by
    \begin{equation}\label{M1}
        MO_{2,r}(f)(z)=\left(  
        \frac{1}{|D^{r}(z)|}\int_{D^{r}(z)}|f-\hat{f}_{r}(z)|^{2}dA
        \right)^{1/2}.
\end{equation}
Given  $0<p\leq\infty$ and $\alpha\in\mathbb{R}$, we define the space $\operatorname{IMO}^{p,2,\alpha}_{r}$ to be the family of those $f\in L^{2}_{loc}$ such that
    \begin{equation}\label{M2}
        \|f\|_{\operatorname{IMO}^{p,2,\alpha}_{r}}=\|\rho^{\alpha}MO_{2,r}(f)\|_{L^{p}}<\infty.
    \end{equation}

\begin{thm}\label{Thm5.4}
    Let $0<p<\infty$ and assume that $\phi\in \mathcal{C}^{\infty}(\mathbb{C})$ is subharmonic such that $d\mu=\Delta \phi dA$ is a doubling measure. Then the following are equivalent.
    \begin{itemize}
        \item[(1)] Both $H_{f}$ and $H_{\Bar{f}}\in S_{p}(F^{2}_{\phi},L^{2}_{\phi})$,
        \item[(2)] $f\in \operatorname{IMO}^{p,2,-2/p}_{r}$, for some (equivalent any) $r>0$.
        Moreover,
        \begin{equation}\label{M29}
            \|H_{f}\|_{S_{p}}+\|H_{\Bar{f}}\|_{S_{p}}\simeq \|f\|_{\operatorname{IMO}^{p,2,-2/p}_{r}}.
        \end{equation}
    \end{itemize}
\end{thm}

Using the preceding result, it is easy to show that $H_{\bar f}$ is not Hilbert-Schmidt on $F^2_{\phi}$ when $f$ is a non-constant entire function (see Theorem~\ref{Schneider_thm}), which implies an analogous result of Schneider~\cite{schneider} for the canonical weights $\phi(z) = |z|^m$ and $f(z)=z^k$ when $k$ is a positive integer and $m>0$. However, when we restrict our study to bounded symbols, it turns out that $H_{\bar f}\in S_2$ whenever $H_f\in S_2$ as seen in the following theorem.

\begin{thm}[Berger-Coburn phenomenon for Hilbert-Schmidt Hankel operators]\label{Thm1.3}
    Let  $\phi\in \mathcal{C}^{\infty}(\mathbb{C})$ be subharmonic and suppose that that $d\mu=\Delta\phi\, dA$ is a doubling measure. Then for $f\in L^{\infty}$, $H_{f}\in S_{2}(F^{2}_{\phi},L^{2}_{\phi})$ if and only if $H_{\Bar{f}}\in S_{2}(F^{2}_{\phi},L^{2}_{\phi})$, with
    \begin{equation}\label{B1}
    \|H_{\bar{f}}\|_{S_{2}}\simeq \|H_{f}\|_{S_{2}}.
    \end{equation}
\end{thm}

It is worth emphasizing that the preceding theorem for Hilbert-Schmidt Hankel operators was proved by Bauer~\cite{bauer2004hilbert} in 2004, and it took almost two decades until it was proved for other Schatten classes by Hu and Virtanen~\cite{Janihu2022schatten}. This leads to the following question.

\begin{OP}
Does the Berger-Coburn phenomenon hold true for other Schatten classes $S_p$ when $1<p<\infty$?
\end{OP}

For a discussion on the preceding open problem (involving the Muckenhoupt condition for the boundedness of the Beurling-Ahlfors operator), see Remark~\ref{Muckenhoupt} in Section 6.

\medskip

Before stating our last theorem, we recall the following growth condition for the function $\rho$. Given a doubling Fock space $F^2_\phi$, 
there are constants $C,\eta>0$ and $0\leq\beta<1$ such that
\begin{equation}\label{e:growth condition}
    C^{-1}|z|^{-\eta}\leq \rho(z)\leq C|z|^{\beta}
    \end{equation}
for $|z|>1$ (see Equation (5) of \cite{marco2003interpolating}); we denote the smallest $\beta$ that satisfies \eqref{e:growth condition} by $\beta_{\phi}$. 

The following result shows the Berger-Coburn phenomenon fails for $S_p(F^2_\phi, L^2_\phi)$ provided that $\beta_\phi$ is sufficiently small in comparison with the value of $p$.

\begin{thm}\label{thm1.4}
    Let $\phi\in \mathcal{C}^{\infty}(\mathbb{C})$ be subharmonic with $d\mu=\Delta\phi\, dA$ a doubling measure. Then, for $0<p\leq 1$ with $\beta_\phi\leq \frac{1-p}{1-p/2}$, the Berger-Coburn phenomenon for Schatten class Hankel operators fails; that is, there is an $f\in L^{\infty}(\mathbb{C})$ such that $H_{f}\in S_{p}(F^{2}_{\phi},L^{2}_{\phi})$ but $H_{\bar{f}}\notin S_{p}(F^{2}_{\phi},L^{2}_{\phi})$.

    In particular, when $\rho$ is bounded, the Berger-Coburn phenomenon fails for all $0<p\le 1$.
\end{thm}

A simple consequence of the preceding theorem is that if $F^2_\phi$ is a doubling Fock space, then the Berger-Coburn phenomenon fails for $S_p(F^2_\phi, L^2_\phi)$ provided that $p$ is sufficiently small.

Another consequence is the following corollary, in which we consider again the canonical doubling weights $\phi(z)=|z|^{m}$ and determine when the Berger-Coburn phenomenon fails for these weights.

\begin{cor}\label{canonical weight cor}
Let $m>0$ and $0<p\le 1$. Then the Berger-Coburn phenomenon fails for $S_p(F^2_{|z|^m}, L^2_{|z|^m})$ if
$$
    m\ge \frac p{1-\tfrac{p}2}.
$$
In particular, if $m\ge 2$, then the phenomenon fails for all Schatten classes $S_p$ with $0<p\le 1$.
\end{cor}

Theorem \ref{thm1.4} and its corollary lead to the following question.

\begin{OP}
Determine whether the Berger-Coburn phenomenon fails for $S_p(F^2_\phi, L^2_\phi)$ when $0<p\le 1$ and $\Delta\phi\, dA$ is doubling.
\end{OP}

The paper is organized as follows.
In the next section, we provide preliminaries on the reproducing kernel, including global and local estimates, and elaborate more on the radius function $\rho$ and the induced metric on the complex plane. In Section 3, we provide useful lemmas and use them to prove Theorem \ref{Thm1.1} (IDA decomposition). 
In Section 4, we use Toeplitz operators with locally finite positive Borel measures to prove Theorem \ref{Thm1.2}, which characterizes the Schatten class membership of Hankel operators. Section 5 is devoted to the study of the function space $\operatorname{IMO}$ of integral mean oscillation, which we use to prove Theorem~\ref{Thm5.4}.  
Finally, in Section 6, we prove the Berger-Coburn phenomenon for Hilbert-Schmidt Hankel operators on general doubling Fock spaces as stated in Theorem \ref{Thm1.3}. We finish the last section with the proofs of Theorem \ref{thm1.4} and Corollary~\ref{canonical weight cor}.

\section{Preliminaries}In this section we recall and prove some key lemmas on the function $\rho$, the reproducing kernel of $F^{2}_{\phi}$, the space $\operatorname{IDA}^{p,q,\alpha}_{r}$, and their related integral and norm estimates.

\textbf{Notation.} We use $C$ to denote positive constants whose value may change from line to line but does not depend on the functions being considered. We say that $A\simeq B$ if there exists a constant $C>0$ such that $C^{-1}A\leq B\leq CA$. Moreover, $A\lesssim B$ if $A\leq CB$ for some positive constant $C$.

Let $\phi$ be a subharmonic function on $\mathbb{C}$ such that $d\mu=\Delta\phi dA$ is a doubling measure. Recall that there is a function $\rho$ such that $\mu(D(z,\rho(z)))=1$, for every point $z\in\mathbb{C}$. In other words, the radius of a disk with unit measure depends on the center of the disk. As shown in the Fig 1, $D(z,\rho(z))\subset D(w,|w-z|+\rho(z))$. Hence, $1\leq\mu(D(w,|w-z|+\rho(z)))$, and thus $\rho(w)\leq \rho(z)+|w-z|$. By symmetry,
\begin{equation}
    |\rho(w)-\rho(z)|\leq |w-z|,\quad\textrm{for every }z,w\in\mathbb{C}.
\end{equation}
\begin{figure}[bth]
\label{Fig1}
\centering
\newcommand{\RA}{2}
\pgfmathsetmacro{\RB}{\RA*(4)}
\begin{tikzpicture}[scale=0.3]
\draw (0,0) circle[radius=\RB cm];
 
\draw[red,fill=red] (0,0) circle (.5ex);
\node at (0,0.3) {$w$};

\foreach	\ang in {180}
\draw (\ang:\RB-\RA) circle[radius=\RA cm];
\draw[-,thick,blue] (0,0) -- (-\RB+\RA,0);
\node at (-\RB+\RA,0.2){$z$};
\node at (-\RB+\RA,1){$\rho(z)$};
\draw[dashed, orange] (-\RB+\RA,0)--++(120:\RA);

\draw[-,green] (0,0) -- ++(45:\RB); 
\node at (2,2){$\rho(z)+|w-z|$};
\end{tikzpicture}   
\caption{Relation between $\rho(z)$ and $\rho(w)$}
\end{figure}
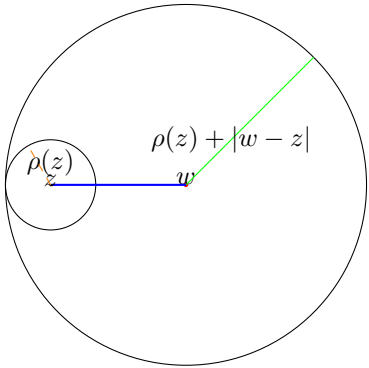
\begin{lem}[See \cite{oliver2016toeplitz}, Lemma 2.2]\label{lem-RadiusEquivalence}
For every $r>0$ there is a constant $c_{r}\geq 1$, depending only on $r$ and the doubling constant for $\mu$, such that
\begin{equation}\label{P2}
    c_{r}^{-1}\rho(z)\leq \rho(w)\leq c_{r}\rho(z), \quad\textrm{for every $z\in\mathbb{C}$ and $w\in D^{r}(z)$}.
\end{equation}
Namely, $c_{r}=(1-r)^{-1}$, for every $0<r<1$. In other words, $\rho(w)$ and $\rho(z)$ are equivalent on a disk.
\end{lem}
Consider the distance $d_{\phi}$ induced by the metric $\rho^{-2}dz\otimes d\bar{z}$. Indeed, for any $z,w\in\mathbb{C}$,
\begin{equation}
    d_{\phi}(z,w)=\inf_{\gamma}\int_{0}^{1}\frac{|\gamma'(t)|}{\rho(\gamma(t))}dt,
\end{equation}
where the infimum is taken over all piecewise $\mathcal{C}^{1}$ curves $\gamma:[0,1]\to \mathbb{C}$ with $\gamma(0)=z$ and $\gamma(1)=w$.
\begin{lem}[See \cite{marco2003interpolating}, Lemma 4]
There exists $\delta>0$ such that for every $r>0$ there exists $C_{r}>0$ such that
\begin{equation}
    C_{r}^{-1}\frac{|z-w|}{\rho(z)}\leq d_{\phi}(z,w)\leq  C_{r}\frac{|z-w|}{\rho(z)},\quad\textrm{for } w\in D^{r}(z),
\end{equation}
and
\begin{equation}
    C_{r}^{-1}\big(\frac{|z-w|}{\rho(z)}\big)^{\delta}\leq d_{\phi}(z,w)\leq  C_{r}\big(\frac{|z-w|}{\rho(z)}\big)^{2-\delta},\quad\textrm{for } w\in\mathbb{C}\setminus D^{r}(z),
\end{equation}
\end{lem}
Now we can state the following pointwise estimate for the Bergman kernel.
\begin{lem}
\begin{itemize}
    \item [(1)] There exist $C,\epsilon>0$ such that
    \begin{equation}
        |K(w,z)|\leq C \frac{e^{\phi(w)+\phi(z)}}{\rho(w)\rho(z)}e^{-\big(\frac{|z-w|}{\rho(z)}\big)^{\epsilon}},
        \quad w,z\in\mathbb{C},
    \end{equation}
    \item[(2)] There exists some $r_{0}>0$ such that for $z\in\mathbb{C}$ and $w\in D^{r_{0}}(z)$, we have
    \begin{equation}
        |K(w,z)|\simeq \frac{e^{\phi(w)+\phi(z)}}{\rho(z)^{2}}.
    \end{equation}
    \item[(3)] $k_{p,z}\to 0$ uniformly on compact subsets of $\mathbb{C}$ as $z\to\infty$, where $k_{p,z}:=\frac{K_{z}}{\|K_{z}\|_{p,\phi}}$ is the normalized Bergman kernel of $F^{p}_{\phi}$.
    \item[(4)] For any $1\leq p\leq\infty$, we have that
    \begin{equation}
        \|K_{z}\|_{p,\phi}\simeq e^{\phi(z)}\rho(z)^{2/p-2}.
    \end{equation}
\end{itemize}
\end{lem}
\begin{proof}
    See Theorem 1.1 and Proposition 2.11 of \cite{marzo2009pointwise} respectively for parts $(1)$ and $(2)$, Lemma 2.3 of \cite{hu2017positive} for part $(3)$, and Proposition 2.9 of \cite{oliver2016toeplitz} for part $(4)$. 
\end{proof}
Given a sequence $\{a_{j}\}_{j=1}^{\infty}\subset\mathbb{C}$, and $r>0$, we call $\{a_{j}\}_{j=1}^{\infty}$ an $r$-lattice if $\{D^{r}(a_{j})\}_{j=1}^{\infty}$ covers $\mathbb{C}$ and the disks of $\{D^{r/5}(a_{j})\}_{j=1}^{\infty}$ are pairwise disjoint. Moreover, for an $r$-lattice $\{a_{j}\}_{j=1}^{\infty}$, and a real number $m>1$, there exists an integer $N$ such that 
\begin{equation}
    1\leq \sum_{j=1}^{\infty}\chi_{D^{mr}(a_{j})(z)}\leq N
\end{equation}
where $\chi_{E}$ is the characteristic function of a subset $E$ of $\mathbb{C}$.
For $f,e\in L^{2}_{\phi}$, the tensor product $f\otimes e$ as a rank one operator on $L^{2}_{\phi}$ is defined by
\begin{equation}
    f\otimes e(g)=\langle g,e\rangle f,\quad g\in L^{2}_{\phi}.
\end{equation}
\begin{lem}\label{lemT:bounded}
Given $r>0$, there is some constant $C>0$ such that if $\Gamma$ is an $r$-lattice in $\mathbb{C}$, and if $\{e_{a}:a\in\Gamma\}$ is an orthonormal set in $L^{2}_{\phi}$, then 
\begin{equation}
    \left\|\sum_{a\in\Gamma}k_{2,a}\otimes e_{a}\right\|_{L^{2}_{\phi}\to L^{2}_{\phi}}\leq C,
\end{equation}
where $k_{2,a}:=\frac{K_{a}}{\|K_{a}\|_{2,\phi}}$ is the normalized Bergman kernel.
\end{lem}
\begin{proof}
Note that $\{\lambda_{a}=\langle g,e_{a}\rangle_{2,\phi}\}_{a\in\Gamma}\in l^{2}$. Then similar to the proof of Lemma 2.4 in \cite{hu2014toeplitz}, 
\begin{equation}
    \left\|\sum_{a\in\Gamma}\lambda_{a}k_{2,a}\right\|\leq C \left\| \{\lambda_{a}\}_{a\in\Gamma} \right\|_{l^{2}},
\end{equation}
where the constant $C$ only depends on $r$. Then similar to the proof of Lemma 2.4 in \cite{Janihu2022schatten}, we have
\begin{equation}
   \left \|\big(\sum_{a\in\Gamma}k_{2,a}\otimes e_{a}\big)(g)\right\|^{2}\leq
    C|\langle g,e_{a}\rangle|^{2}\leq C\|g\|^{2}.
\end{equation}
\end{proof}

We finish this section with a description of $\rho$ for the canonical weights $|z|^m$ with $m>0$.
\begin{lem}\label{canonical weight lem}
Let $\phi(z) = |z|^m$ with $m>0$. Then $d\mu=\Delta\phi dA$ is a doubling measure. Moreover, there is an $R>0$ such that
$$
    \rho(z) \simeq |z|^{1-m/2}
$$
for $|z|>R$. In particular, when $m\ge 2$, $\rho$ is bounded.
\end{lem}
\begin{proof}
Note that $\Delta\phi(z)=m^{2}|z|^{m-2}$. To show that $d\mu$ is a doubling weight, it is enough to prove that for any $x\geq 0$ and $r>0$, 
\begin{equation}\label{C0}
    \int_{D(x,2r)}|z|^{m-2} dA(z)\leq C \int_{D(x,r)}|z|^{m-2} dA(z),
\end{equation}
where the constant $C$ is independent of $x$ and $r$.

We consider $r>\frac{x}{100}\geq 0$ first. Then $D(x,2r)\subset D(0,x+2r)$, so that
\begin{equation}\label{C1}
    \int_{D(x,2r)}d\mu(\xi)\leq \int_{|\xi|\leq x+2r}|\xi|^{m-2}dA(\xi)
    \leq \int_{|\xi|\leq 102r}|\xi|^{m-2}dA(\xi) \leq C_{1}r^{m}.
\end{equation}
On the other hand, if $m\geq 2$,
\begin{equation}\label{C2}
    \int_{D(x,r)}d\mu(\xi)\geq \int_{D(x,r)\cap \{\operatorname{Re}\xi\geq x\}}d\mu(\xi)
    \geq \int_{D(0,r)\cap \{\operatorname{Re}\xi\geq 0\}}d\mu(\xi)
    \geq C_{2}r^{m}.
\end{equation}
From (\ref{C1}) and (\ref{C2}) we obtain (\ref{C0}) for $m\geq 2$ and $r>\frac{x}{100}$. 

Now we suppose $0<r<\frac{x}{100}$. Then
\begin{align*}
   & D(x,2r)\subset \{te^{i\theta}: x-2r<t<x+2r, |\theta|<\operatorname{arcsin}\frac{2r}{x}\},
   \\
   &
   D(x,r)\supset \{te^{i\theta}: x-c_{1}r<t<x+c_{2}r, |\theta|<\operatorname{arcsin}\frac{r}{2x}\},
\end{align*}
where $c_{1}$ and $c_{2}$ are positive constants independent of $x$ are $r$. Hence,
\begin{align}\label{C3}
   \int_{D(x,2r)}d\mu&\leq \int_{x-2r}^{x+2r} r^{m-1}dr\int_{-\operatorname{arcsin}\frac{2r}{x}}^{\operatorname{arcsin}\frac{2r}{x}}d\theta
   \simeq \frac{r}{x}\left[(x+2r)^{m}-(x-2r)^{m} \right]\\\nonumber
   &\simeq \frac{r}{x}rx^{m-1}=r^{2}x^{m-2}.
\end{align},
where the constant in the inequalities $\simeq$ are all independent of $x$ and $r$. Similarly,
\begin{align}\label{C4}
   \int_{D(x,r)}d\mu&\geq \int_{x-c_{1}r}^{x+c_{2}r} r^{m-1}dr\int_{-\operatorname{arcsin}\frac{r}{2x}}^{\operatorname{arcsin}\frac{r}{2x}}d\theta\\\nonumber
   &
   \simeq \frac{r}{x}\left[(x+c_{2}r)^{m}-(x-c_{1}r)^{m} \right]
   \simeq r^{2}x^{m-2}.
\end{align}
Using (\ref{C3}) and (\ref{C4}), we obtain (\ref{C0}). 

For $0<m<2$, and $r>\frac{x}{100}$, 
\begin{equation}\label{C5}
    \int_{D(x,r)}|\xi|^{m-2}dA(\xi)
    = \int_{D(0,r)}|\xi+x|^{m-2}dA(\xi)
    \geq \int_{D(0,r)}|\xi|^{m-2}dA(\xi)
    \geq C_{3} r^{m}.
\end{equation}
From (\ref{C1}) and (\ref{C5}) we obtain (\ref{C0}) for $0<m<2$ and $r>\frac{x}{100}$. 

Now notice that using (\ref{C3}) and (\ref{C4}) and when $x$ is large enough, 
\begin{equation}
    \int_{(x,x^{-\frac{m-2}{2}})}|\xi|^{m-2}dA(\xi)\simeq 1.
\end{equation}
This, together with the doubling property implies that there exists $R>0$ large enough, such that for the Fock space $F^{2}_{|z|^{m}}$,
\begin{equation}
    \rho(z)\simeq |z|^{-\frac{m-2}{2}}=|z|^{1-\frac{m}{2}}
\end{equation}
for $|z|\geq R$.
\end{proof}

\section{The Space IDA}
The goal of this section is to prove the IDA decomposition Theorem \ref{Thm1.1}. Before proving the theorem, we need to see some definitions and lemmas. 
\begin{lem}\label{lemIDA1}
    Suppose $1\leq q<\infty$. Then for $f\in L^{q}_{loc}$, $z\in\mathbb{C}$, and $r>0$, there is $h\in H(D^{r}(z))$ such that
    \begin{equation}\label{D4}
        \big(\widehat{|f-h|^{q}}_{r} (z)\big)^{1/q}=G_{q,r}(f)(z),
    \end{equation}
    and for $s<r$,
    \begin{equation}\label{D5}
        \sup_{w\in D^{s}(z)} |h(w)|\leq C \|f\|_{L^{q}(D^{r}(z), dA)},
    \end{equation}
    where the constant $C$ is independent of $f$ and $r$. 
\end{lem}
\begin{proof}
    This proof is similar to the proof of Lemma 3.3 in \cite{Jani2021ida}. Taking $h=0$, 
    \begin{equation}
        G_{q,r}(f)(z)\leq \big(\widehat{|f|^{q}}_{r}(z) \big)^{1/q}< \infty.
    \end{equation}
    Then for $j=1,2,\cdot\cdot\cdot$, pick $h_{j}\in H(D^{r}(z))$ such that
    \begin{equation}\label{D7}
        \big(\widehat{|f-h_{j}|^{q}}_{r}(z) \big)^{1/q} \to G_{q,r}(f)(z) \quad \textrm{as } j\to\infty.
    \end{equation}
    Hence for sufficiently large $j$,
 \begin{equation}
     \big(\widehat{|h_{j}|^{q}}_{r} (z)\big)^{1/q} \leq
     C\{\big(\widehat{|f-h_{j}|^{q}}_{r} (z)\big)^{1/q}+\big(\widehat{|f|^{q}}_{r} (z)\big)^{1/q}\} \leq C\big(\widehat{|f|^{q}}_{r} (z)\big)^{1/q}.
 \end{equation}
 Thus, we can find a subsequence $\{h_{j_{k}}\}_{k=1}^{\infty}$ and a function $h\in H(D^{r}(z))$ such that $\lim_{k\to\infty}h_{j_{k}}(w)=h(w)$ for $w\in D^{r}(z)$. By (\ref{D7}),
 \begin{equation}
     G_{q,r}(f)(z)\leq \big(\widehat{|f-h|^{q}}_{r} (z)\big)^{1/q} \leq \liminf_{k\to\infty} \big(\widehat{|f-h_{j_{k}}|^{q}}_{r} (z)\big)^{1/q}=G_{q,r}(f)(z)
 \end{equation}
 where in the RHS inequality we have used Fatou's Lemma. This gives us (\ref{D4}).\\
 Now for $w\in D^{s}(z)$, by the mean value Theorem,
 \begin{equation}
     |h(w)|\leq \big(\widehat{|h|^{q}}_{s} (z)\big)^{1/q} \leq C \big(\widehat{|h|^{q}}_{r} (z)\big)^{1/q} \leq \big(\widehat{|f|^{q}}_{r} (z)\big)^{1/q}=C \|f\|_{L^{q}(D^{r}(z), dA)}.
 \end{equation}
\end{proof}
Now we are ready to define $f_{1}$ and $f_{2}$ in Theorem 1.1. Using (\ref{P2}) and the triangle inequality, there exists $m\in (0,1)$ such that $D^{mr}(w)\subset D^{r}(z)$, whenever $w\in D^{mr}(z)$. For $r>0$, let $\{a_{j}\}_{j=1}^{\infty}$ be a $mr$-lattice, and let $J_{z}:=\{j:z\in D^{r}(a_{j})\}$, so that $|J_{z}|=\sum_{j=1}^{\infty}\chi_{D^{r}(a_{j})}(z)\leq N$, for some integer $N$. Let $\eta:\mathbb{C}\to [0,1]$ be the following smooth function with bounded derivatives.
\begin{equation}\label{eta}
    \eta(z) = \left\{
     \begin{array}{@{}l@{\thinspace}l}
       1\quad  &\textrm{if } |z|\leq 1/2,\\
       0\quad &\textrm{if }  |z|\geq 1.\\
     \end{array}
   \right.
\end{equation}
For each $j\geq 1$ we define $\eta_{j}(z)=\eta(\frac{z-a_{j}}{mr\rho(a_{j})})$. We can normalize $\eta_{j}$ such that $\int_{\mathbb{C}}\eta_{j}dA=1$, for each $j\geq 1$. Define $\psi_{j}(z)=\frac{\eta_{j}(z)}{\sum_{k=1}^{\infty}\eta_{k}(z)}$. Then one can see that $\{\psi_{j}\}_{j=1}^{\infty}$ is a partition of unity subordinate to $\{D^{mr}(a_{j})\}_{j\geq 1}$,  satisfying the following properties: 
\begin{align}\label{D11}
\operatorname{Supp}\psi_{j}\subset D^{mr}(a_{j}),\quad \psi_{j}(z) \geq 0,\quad \sum_{j=1}^{\infty}\psi_{j}(z)=1,\nonumber\\ |\rho(a_{j})\Bar{\partial}\psi_{j}|\leq C,\quad \sum_{j=1}^{\infty}\Bar{\partial}\psi_{j}(z)=0,
\end{align}
where the constant $C$ may depend on $r$.

By Lemma \ref{lemIDA1}, for $j=1,2,\cdot\cdot\cdot$, we can pick $h_{j}\in H(D^{r}(a_{j}))$ such that
\begin{equation}
   \widehat{|f-h_{j}|^{q}}_{r} (a_{j}) =\frac{1}{|D^{r}(a_{j})|}\int_{D^{r}(a_{j})}|f-h_{j}|^{q}dA=G_{q,r}(f)(a_{j})^{q}.
\end{equation}
For $1\leq q<\infty$ and $f\in L^{q}_{loc}$, decompose $f=f_{1}+f_{2}$ as
    \begin{equation}\label{D13}
        f_{1}(z):=\sum_{j=1}^{\infty}h_{j}(z)\psi_{j}(z), \quad f_{2}(z):=f(z)-f_{1}(z).
    \end{equation}
\begin{lem}\label{lemIDA2}
    Let $1\leq q<\infty$, $f\in L^{q}_{loc}$, and $r>0$. Decomposing $f=f_{1}+f_{2}$ as in (\ref{D13}), we have $f_{1}\in \mathcal{C}^{2}(\mathbb{C})$ and
    \begin{equation}\label{D14}
        \rho(z)|\Bar{\partial} f_{1}(z)|+\rho(z)(\widehat{|\Bar{\partial} f_{1}|^{q}}_{mr})^{1/q}+ (\widehat{| f_{2}|^{q}}_{mr})^{1/q} \leq CG_{q,R}(f)(z),
    \end{equation}
    for some $R>r$ and $m\in (0,1)$.
\end{lem}
\begin{proof}
    Using the properties of $h_{j}$ and $\psi_{j}$ we can easily see that $f_{1}\in \mathcal{C}^{2}(\mathbb{C})$. Let $z\in\mathbb{C}$, and $J_{z}=\{j:z\in D^{r}(a_{j})\}$. We know that if $z\in D^{r}(a_{j})$, then $\rho(z)\leq C\rho(a_{j})$. Therefore, knowing $\sum_{j=1}^{\infty}\Bar{\partial}\psi_{j}=0$, using (\ref{D11}), the triangle inequality, and since $|h_{j}-h_{1}|^{q}$ is plurisubharmonic on $D^{r}(a_{j})$,
    \begin{align}\label{D15}
          \rho(z)|\Bar{\partial}f_{1}(z)|
          &
           =\rho(z) \left|\Bar{\partial}(\sum_{j=1}^{\infty}h_{j}(z)\psi_{j}(z))\right|
            \leq
            \rho(z) \sum_{j=1}^{\infty}|h_{j}(z)-h_{1}(z)||\Bar{\partial}\psi_{j}(z)|\nonumber\\
          &\leq C\sum_{j\in J_{z}}\left[\frac{1}{|D^{r}(a_{j})|}\int_{D^{r}(a_{j})}|h_{j}-h_{1}|^{q}dA\right]^{1/q}\rho(a_{j})|\Bar{\partial}\psi_{j}(z)|\nonumber
          \\
          &\leq C\sum_{j\in J_{z}}\left[\frac{1}{|D^{r}(a_{j})|}\int_{D^{r}(a_{j})}\{|f-h_{j}|^{q}+|f-h_{1}|^{q}\}dA\right]^{1/q}\nonumber \\
          &\leq C\sum_{j\in J_{z}}\big(\widehat{|f-h_{j}|^{q}}_{r}(a_{j})\big)^{1/q}+ \big(\widehat{|f-h_{1}|^{q}}_{r}(a_{j})\big)^{1/q}\nonumber
          \\
          &\leq C \sum_{j\in J_{z}}G_{q,r}(a_{j})\leq C G_{q,s}(f)(z),
    \end{align}
    for some $s>r$, where the last inequality can be shown similarly to Corollary 3.4 in \cite{Jani2021ida}, and using the fact that $|J_{z}|$ is finite.
    
    Moreover, note that
    \begin{align}\label{D17}
\rho(z)\big(\widehat{|\Bar{\partial}f_{1}|^{q}}_{mr}(z)\big)^{1/q}&=
         \rho(z)\left[\frac{1}{|D^{mr}(z)|}\int_{D^{mr}(z)}|\Bar{\partial}f_{1}(w)|^{q}dA(w)\right]^{1/q}\nonumber \\
         &\leq
         C \left[\frac{1}{|D^{mr}(z)|}\int_{D^{mr}(z)}\rho(w)^{q}|\Bar{\partial}f_{1}(w)|^{q}dA(w)\right]^{1/q}\nonumber \\
         &\leq
         C \left[\frac{1}{|D^{mr}(z)|}\int_{D^{mr}(z)}G_{q,s}(f)(w)^{q}dA(w)\right]^{1/q}\nonumber\\
         &\leq
         C\sup_{w\in D^{mr}(z)}G_{q,s}(f)(w)\leq CG_{q,R}(f)(z),
    \end{align}
    for some $R>s$, where again for the last inequality we use Corollary 3.4 in \cite{Jani2021ida}.
    Similarly, since $\sum_{j=1}^{\infty}\psi_{j}=1$,
    \begin{equation}
        |f_{2}(w)|^{q}=|f(w)-\sum_{j=1}^{\infty}h_{j}(w)\psi_{j}(w)|^{q}
        \leq
        \sum_{j=1}^{\infty}|f(w)-h_{j}(w)|^{q}|\psi_{j}(w)|^{q}.
    \end{equation}
    Hence, using $|\psi_{j}|\leq 1$,
    \begin{align}
            \big(\widehat{|f_{2}|^{q}}_{mr}(z)\big)^{1/q} &\leq
            \sum_{j=1}^{\infty} \left[\frac{1}{|D^{mr}(z)|}\int_{D^{mr}(z)}|f-h_{j}|^{q}|\psi_{j}|^{q}dA\right]^{1/q}\nonumber\\
            &\leq 
            C\sum_{j\in J_{z}}G_{q,r}(f)(a_{j})\leq CG_{q,R}(f)(z),
    \end{align}
    similar to the previous part for $\rho|\Bar{\partial}f_{1}|$. Putting everything together, we can find a big enough $R>r$ such that (\ref{D14}) holds.
\end{proof}

\begin{proof}[Proof of Theorem \ref{Thm1.1}]
First, we show that if (\ref{D1}) holds for some $r$, then it holds for any $r$. 
Let $R>0$. For $0<r<R$ take $t=\frac{r}{2C_{2}R}$ and take $z_{1},\cdot\cdot\cdot,z_{N}$ in the unit disk $D(0,1)$ so that $D(0,1)\subset \cup_{j=1}^{N}D(z_{j},t)$. Set $a_{j}(z)=z+R\rho(z)z_{j}$. Then
\begin{align}
    D^{R}(z)&\subset \cup_{j=1}^{N}D(z+R\rho(z)z_{j},tR\rho(z)) \subset \cup_{j=1}^{N}D(a_{j}(z),\frac{r}{2}\rho(a_{j}(z)))\nonumber\\
    &= \cup_{j=1}^{N}D^{r/2}(a_{j}(z)).
\end{align}
Therefore,
\begin{align}
\int_{\mathbb{C}}\big(\widehat{|g|^{q}}_{R}(z)\big)^{s} dA(z)& \leq C \int_{\mathbb{C}}\sum_{j=1}^{N}\big(\widehat{|g|^{q}}_{r/2}(a_{j}(z))\big)^{s} dA(z) \nonumber\\
    &\leq  C \int_{\mathbb{C}}dA(z)\sum_{j=1}^{N}\frac{1}{|D^{cr}(a_{j}(z))|}\int_{D^{cr}(a_{j}(z))}
    \big(\widehat{|g|^{q}}_{r}(u)\big)^{s} dA(u)\nonumber\\
    &=  C \int_{\mathbb{C}}\big(\widehat{|g|^{q}}_{r}(u)\big)^{s} dA(u)
\sum_{j=1}^{N} \int_{\mathbb{C}}\chi_{D^{cr}(a_{j}(z))}(u)\frac{1}{|D^{cr}(a_{j}(z))|}dA(z)\nonumber\\
&\leq C \int_{\mathbb{C}}\big(\widehat{|g|^{q}}_{r}(u)\big)^{s} dA(u),
\end{align}
where for the second inequality take $c>0$ such that $D^{cr}(a_{j}(z))\subset \cap_{u\in D^{cr}(a_{j}(z)
)}D^{r}(u)$.
Taking $s=p/q$ implies that (\ref{D1}) holds for some $r>0$, if and only if it holds for any $r$.

Now assume that $f\in\operatorname{IDA}^{p,q,\alpha}_{r}$. That is, $f\in L^{q}_{loc}$ with $\|\rho^{\alpha}G_{q,r}(f)\|_{L^{p}}< \infty$. Decompose $f=f_{1}+f_{2}$ as in Lemma \ref{lemIDA2}. Then $f_{1}\in \mathcal{C}^{2}(\mathbb{C})$, and (\ref{D14}) holds. Multiplying both sides with $\rho^{\alpha}$ and taking the $L^{p}$-norm, we obtain (\ref{D1}).

\end{proof}

\section{Schatten Class Hankel Operators on Doubling Fock Spaces}
Recall that for a bounded linear operator $T:H_{1}\to H_{2}$ between two Hilbert spaces, the singular values $\lambda_{n}$ are defined by 
\begin{equation}
    \lambda_{n}=\lambda_{n}(T)=\inf \{\|T-K\|:K:H_{1}\to H_{2}, \operatorname{rank}K\leq n
    \}.
\end{equation}
The operator $T$ is compact if and only if $\lambda_{n}\to 0$. Given $0<p<\infty$, we say that $T$ is in the Schatten class $S_{p}$ and write $T\in S_{p}(H_{1},H_{2})$, if its singular value sequence $\{\lambda_{n}\}$ belongs to $l^{p}$. Then $\|T\|^{p}_{S_{p}}=\sum_{n=0}^{\infty}|\lambda_{n}|^{p}$ defines a norm when $1\leq p<\infty$ and a quasinorm when $0<p<1$. Moreover, for the quasi-Banach case, we have the triangle inequality.
\begin{equation}
    \|T+S\|^{p}_{S_{p}}\leq \|T\|^{p}_{S_{p}}+\|S\|^{p}_{S_{p}},\quad \textrm{when } T,S\in S_{p},\textrm{ }0<p<1,
\end{equation}
which is called the Rotfel'd inequality. For a positive compact operators $T$ on $H$ and $p>0$, $T\in S_{p}$ if and only if $T^{p}\in S_{1}$. Moreover, $\|T\|^{p}_{S_{p}}=\|T^{p}\|_{S_{1}}$. See \cite{zhu2007operator} for further details on the properties of Schatten class operators, as well as the proof of the next two theorems.
\begin{thm}[See \cite{zhu2007operator}, Theorem 1.26]\label{thmkehe1}
If $T$ is a compact operator on $H$ and $p>0$, then $T\in S_{p}$ if and only if $|T|^{p}=(T^{*}T)^{p/2}\in S_{1}$, if and only if $T^{*}T\in S_{p/2}$. Moreover,
\begin{equation}
\|T\|^{p}_{S_{p}}=\||T|\|^{p}_{S_{p}}=\||T|^{p}\|_{S_{1}}=\|T^{*}T\|^{p/2}_{S_{p/2}}.
\end{equation}
Consequently, $T\in S_{p}$ if and only if $|T|\in S_{p}$.
\end{thm}
\begin{thm}[See \cite{zhu2007operator}, Theorem 1.28]\label{thmkehe2}
    Suppose $T$ is a compact operator on $H$ and $p\geq 1$. Then  $T$ is in $S_{p}$ if and only if 
    \begin{equation}
        \sum |\langle Te_{n},\sigma_{n}
        \rangle|^{p}<\infty,
    \end{equation}
    for all orthonormal sets $\{e_{n}\}$ and $\{\sigma_{n}\}$. If $T$ is positive, we also have
    \begin{equation}
        \|T\|_{S_{p}}=\sup \left\{ \big[\sum |\langle Te_{n},\sigma_{n}
        \rangle|^{p}\big]^{1/p}:\quad \{e_{n}\}\textrm{ and }\{\sigma_{n}\}\textrm{ are orthonormal}
       \right \}.
    \end{equation}
\end{thm}
Given a locally finite positive Borel measure $\mu$ on $\mathbb{C}$, we define the Toeplitz operator $T_{\mu}$ with symbol $\mu$ as
\begin{equation}
    T_{\mu}f(z)=\int_{\mathbb{C}} f(w)\overline{K_{z}(w)}e^{-2\phi(w)}d\mu(w).
\end{equation}
Moreover, for every $r>0$, the $r$-averaging transform of $\mu$ is defined by
\begin{equation}
    \hat{\mu}_{r}(z):=\frac{\mu(D^{r}(z))}{|D^{r}(z)|}\simeq \frac{\mu(D^{r}(z))}{\rho(z)^{2}}.
\end{equation}
\begin{thm}[See \cite{oliver2016toeplitz}, Theorem 4.1]\label{Thm3.2}
    Let $\mu$ be a locally finite positive Borel measure on $\mathbb{C}$, and let $0<p<\infty$. Then the following are equivalent.
    \begin{itemize}
        \item[(1)] $T_{\mu}\in S_{p}(F^{2}_{\phi})$,\item[(2)] There is $r_{0}>0$ such that any $r$-lattice $\{z_{j}\}_{j\geq 1}$ with $r\in (0,r_{0})$ satisfies $\{\hat{\mu}_{r}(z_{j})\}_{j\geq 1}\in l^{p}$,
        \item[(3)] There is an $r$-lattice $\{z_{j}\}_{j\geq 1}$ such that $\{\hat{\mu}_{r}(z_{j})\}_{j\geq 1}\in l^{p}$,
        \item[(4)] There is $r>0$ such that $\hat{\mu}_{r}\in L^{p}(\mathbb{C},d\sigma)$, 
    \end{itemize}
Moreover, $\|T_{\mu}\|^{p}_{S_{p}}\simeq \|\hat{\mu}_{r}\|_{L^{p}(\mathbb{C},d\sigma)}$,
 where $d\sigma=dA/\rho^{2}$.
\end{thm}
The rest of this section is devoted to the proof of the Schatten class membership of the Hankel operators Theorem \ref{Thm1.2}. For this purpose, let $a\in\mathbb{C}$ and $r>0$. Let $A^{2}(D^{r}(a),e^{-2\phi}dA)$ be the weighted Bergman space containing the holomorphic functions in $L^{2}(D^{r}(a),e^{-2\phi}dA)$. Let $P_{a,r}:L^{2}(D^{r}(a),e^{-2\phi}dA)\to A^{2}(D^{r}(a),e^{-2\phi}dA)$ be the orthogonal projection, and for $f\in L^{2}(D^{r}(a),e^{-2\phi}dA)$, extend $P_{a,r}(f)$ to $\mathbb{C}$ by setting
\begin{equation}
    P_{a,r}(f)_{|\mathbb{C}\setminus D^{r}(a)}=0.
\end{equation}
One can check that for $f,g\in L^{2}_{\phi}$,
\begin{equation}
    P^{2}_{a,r}(f)=P_{a,r}(f),\quad\textrm{and }\quad \langle f,P_{a,r}(g)\rangle = \langle P_{a,r}(f),g\rangle.
\end{equation}
Moreover, for $h\in F^{2}_{\phi}$,
\begin{equation}
    P_{a,r}(h)=\chi_{D^{r}(a)}h,\quad\textrm{and }\quad \langle h, \chi_{D^{r}(a)}f-P_{a,r}(f)\rangle=0.
\end{equation}
\begin{proof}[Proof of Theorem \ref{Thm1.2}] Here we borrow an idea from the proof of Proposition 6.8 in \cite{fangXiaProof2018hankel} and the proof of Theorem 1.1 in \cite{Janihu2022schatten}. 
    First we show that $(2)\implies (1)$. Let $f\in\operatorname{IDA}^{p,2,-2/p}_{r}$. Then by Theorem \ref{Thm1.1}, $f=f_{1}+f_{2}$ with
    \begin{equation}\label{S1}
        \rho^{1-2/p}|\Bar{\partial} f_{1}|+\rho^{1-2/p}(\widehat{|\Bar{\partial} f_{1}|^{2}}_{r})^{1/2}+ \rho^{-2/p}(\widehat{| f_{2}|^{2}}_{r})^{1/2} \in L^{p}
    \end{equation}
    Applying the definition,
    \begin{equation}\label{S2}
      \rho^{1-2/p}(z)(\widehat{|\Bar{\partial} f_{1}|^{2}}_{r}(z))^{1/2}=\rho^{1-2/p}(z)\big\{ \frac{1}{|D^{r}(z)|}\int_{D^{r}(z)}|\Bar{\partial}f_{1}|^{2}dA\big\}^{1/2}, 
    \end{equation}
    and
    \begin{equation}\label{S3}
      \rho^{-2/p}(z)(\widehat{| f_{2}|^{2}}_{r}(z))^{1/2}=\rho^{-2/p}(z)\big\{ \frac{1}{|D^{r}(z)|}\int_{D^{r}(z)}|f_{2}|^{2}dA\big\}^{1/2}, 
    \end{equation}
Set $\Phi:=\rho|\Bar{\partial}f_{1}|$ or $\Phi=|f_{2}|$, and $\mu:=|\Phi|^{2}$. First, if $\Phi=\rho|\Bar{\partial}f_{1}|$,
\begin{equation}
    \hat{\mu}_{r}(z):=\frac{\mu(D^{r}(z))}{|D^{r}(z)|}= \frac{1}{|D^{r}(z)|}\int_{D^{r}(z)}|\Phi|^{2}dA
    =\frac{1}{|D^{r}(z)|}\int_{D^{r}(z)}\rho^{2}|\Bar{\partial}f_{1}|^{2}dA.
\end{equation}
We claim that for $f\in\operatorname{IDA}^{p,2,-2/p}_{r}$, $\hat{\mu}_{r}\in L^{p/2}(\mathbb{C},d\sigma)$. Note that
\begin{align}\label{S6}
        \|\hat{\mu}_{r}\|^{p/2}_{L^{p/2}(\mathbb{C},d\sigma)}&= \int_{\mathbb{C}}|\hat{\mu}_{r}|^{p/2}dA/\rho^{2}\nonumber\\
        &= \int_{\mathbb{C}}\frac{1}{|D^{r}(z)|^{p/2}}\big[\int_{D^{r}(z)}\rho^{2}|\Bar{\partial}f_{1}|^{2}dA\big]^{p/2}\frac{dA(z)}{\rho(z)^{2}}. 
\end{align}
Since $f\in\operatorname{IDA}^{p,2,-2/p}_{r}$, we have $\rho^{1-2/p}(\widehat{|\Bar{\partial} f_{1}|^{2}}_{r})^{1/2} \in L^{p}$ and thus
\begin{equation}\label{S7}
    \int_{\mathbb{C}}\rho^{p-2}\big\{ \frac{1}{|D^{r}(z)|}\int_{D^{r}(z)}|\Bar{\partial}f_{1}|^{2}dA\big\}^{p/2}  dA(z)<\infty.
\end{equation}
Recall that in (\ref{S6}), $w\in D^{r}(z)$, and therefore there is a constant $C$ such that $\rho(w)\leq C\rho(z)$. Hence,
\begin{equation}\label{S8}
    \|\hat{\mu}_{r}\|^{p/2}_{L^{p/2}(\mathbb{C},d\sigma)}\leq 
    \int_{\mathbb{C}} \frac{C\rho(z)^{p-2}}{|D^{r}(z)|^{p/2}}\big\{\int_{D^{r}(z)}|\Bar{\partial}f_{1}|^{2}dA\big\}^{p/2}  dA(z) \asymp \textrm{LHS of (\ref{S7})}<\infty.
\end{equation}
Thus, we can conclude that $\hat{\mu}_{r}\in L^{p/2}(\mathbb{C},d\sigma)$, for $\mu=\rho^{2}|\Bar{\partial}f_{1}|^{2}$. Now, using Theorem \ref{Thm3.2}, $T_{\mu}\in S_{p/2}(F^{2}_{\phi})$.

Consider the multiplication $M_{\Phi}:F^{2}_{\phi}\to L^{2}_{\phi}$ defined by $M_{\Phi}f:=\Phi f$. Then $M_{\Phi}$ is bounded for $\Phi=\rho |\Bar{\partial}f_{1}|$ or $\Phi=|f_{2}|$. For $h,g\in L^{2}_{\phi}$,
\begin{equation}
    \langle M_{\Phi}^{*}M_{\Phi}g,h \rangle_{2,\phi}= \langle M_{\Phi}g,M_{\Phi}h \rangle_{2,\phi}= \int_{\mathbb{C}}g\Bar{h}e^{-2\phi}dA=\langle T_{|\Phi|^{2}}g,h \rangle_{2,\phi}.
\end{equation}
so, $M_{\Phi}^{*}M_{\Phi}=T_{|\Phi|^{2}}\in S_{p/2}$, and thus $M_{\Phi}\in S_{p}$. Moreover,
\begin{equation}
    \|M_{\Phi}\|_{S_{p}}\simeq \|M_{\Phi}^{*}M_{\Phi}\|_{S_{p/2}}\simeq \|T_{\mu}\|_{S_{p/2}}\simeq \|\hat{\mu}_{r}\|_{L^{p/2}(\mathbb{C},d\sigma)}.
\end{equation}
By equations (3.13) and (3.17) in \cite{lv2023hankel}, and using Fock-Carleson measures for $F^{2}_{\phi}$, we can see that
\begin{equation}\label{BoundHanhelnorm}
    \|H_{f_{1}}g\|_{2,\phi}\leq \|\rho g\Bar{\partial}f_{1}\|_{2,\phi},\quad \textrm{ and }\quad
    \|H_{f_{2}}g\|_{2,\phi}\leq \|gf_{2}\|_{2,\phi}.
\end{equation}
Therefore,
\begin{equation}
    \|H_{f_{1}}\|_{S_{p}}\lesssim \|M_{\Phi}\|_{S_{p}} \simeq \|\hat{\mu}_{r}\|_{L^{p/2}(\mathbb{C},d\sigma)}\lesssim \|\rho^{1-2/p}(\widehat{|\Bar{\partial}f_{1}|^{2}}_{r})^{1/2}\|_{L^{p}}\asymp \|f\|_{\operatorname{IDA}^{p,2,-2/p}_{r}}.
\end{equation}

To complete the proof, it remains to note that when $\mu=|f_{2}|^{2}$, we have 
\begin{align}
    \|\hat{\mu}_{r}\|^{p/2}_{L^{p/2}(\mathbb{C},d\sigma)}&=  \int_{\mathbb{C}}\frac{1}{|D^{r}(z)|^{p/2}}\big[\int_{D^{r}(z)}|f_{2}|^{2}dA\big]^{p/2}\frac{dA(z)}{\rho(z)^{2}} \nonumber\\
    &= \int_{\mathbb{C}} \big[\frac{\rho(z)^{-2/p}}{|D^{r}(z)|^{1/2}}\{\int_{D^{r}(z)}|f_{2}|^{2}dA\}^{1/2}\big]^{p}dA(z)\nonumber\\
    &=\|\rho^{-2/p}(\widehat{|f_{2}|^{2}}_{r})^{1/2}\|_{L^{p}},
\end{align}
so that 
\begin{equation*}
    \|H_{f_{2}}\|_{S_p}\lesssim \|f\|_{\operatorname{IDA}^{p,2,-2/p}_{r}}.
\end{equation*}
Consequently, $\|H_{f}\|_{S_{p}}\lesssim \|H_{f_{1}}\|_{S_{p}}+\|H_{f_{2}}\|_{S_{p}}\lesssim  \|f\|_{\operatorname{IDA}^{p,2,-2/p}_{r}}$, and so $H_{f}\in S_{p}(F^{2}_{\phi},L^{2}_{\phi})$.

To show $(1)\implies (2)$ for $p\geq 1$, we proceed as follows. Recall that $\{a_{j}\}_{j=1}^{\infty}$ is an $r$-lattice if $\{D^{r}(a_{j})\}_{j=1}^{\infty}$ covers $\mathbb{C}$ and $D^{r/5}(a_{j})\cap D^{r/5}(a_{k})=\emptyset$ for $j\neq k$. Let $\Gamma$ be an $r$-lattice, and let $\{e_{a}:a\in\Gamma\}$ be an orthonormal basis of $F^{2}_{\phi}$. Define linear operators $T$ and $B$ by
\begin{equation}
    T=\sum_{a\in\Gamma}k_{2,a}\otimes e_{a},\quad \textrm{and}\quad B=\sum_{a\in\Gamma}g_{a}\otimes e_{a},
\end{equation}
 where
 \begin{equation}  
g_{a} = 
     \begin{cases}
       \frac{\chi_{D^{r}(a)}H_{f}(k_{2,a})}{\|\chi_{D^{r}(a)}H_{f}(k_{2,a})\|} &\quad\text{if } \|\chi_{D^{r}(a)}H_{f}(k_{2,a})\|\neq 0,\\
       0 &\quad\text{if } \|\chi_{D^{r}(a)}H_{f}(k_{2,a})\|=0.\\ 
     \end{cases}
 \end{equation}
Since $\|g_{a}\|\leq 1$ and $\langle g_{a},g_{b}\rangle=0$ when $a\neq b$, $\|B\|_{L^{2}_{\phi}}\to L^{2}_{\phi}\leq 1$. Moreover, by Lemma \ref{lemT:bounded}, we can see that $\|T\|\leq C$ for some constant $C$. Let $H_{f}\in S_{p}$. So in particular, $H_{f}$ is compact. We know from Lemma 2.3 that $k_{p,z}\to 0$ uniformly on compact subsets of $\mathbb{C}$ as $z\to\infty$, where $k_{p,z}=K_{z}/\|K_{z}\|_{p,\phi}$ is the normalized Bergman kernel for $F^{p}_{\phi}$. By compactness of $H_{f}$ we obtain that
\begin{equation}\label{limcp}
    \lim_{z\to\infty} \|\chi_{D^{r}(z)}H_{f}(k_{2,z})\|_{L^{2}_{\phi}}=0.
\end{equation}
Note that
\begin{align}
        \langle B^{*}M_{\chi_{D^{r}(a)}}H_{f}Te_{a},e_{a}\rangle&= 
        \langle \chi_{D^{r}(a)} H_{f}\sum_{b\in\Gamma}k_{2,b}\otimes e_{b}(e_{a}),\sum_{d\in\Gamma}g_{d}\otimes e_{d}(e_{a})\rangle\nonumber\\
        &= \langle \chi_{D^{r}(a)}H_{f}(k_{2,a}),g_{a}\rangle= \|\chi_{D^{r}(z)}H_{f}(k_{2,z})\|_{L^{2}_{\phi}},
\end{align}
and
\begin{equation}
  \langle B^{*}M_{\chi_{D^{r}(a)}}H_{f}Te_{a},e_{b}\rangle= 0, \quad a\neq b.  
\end{equation}
Thus,  $B^{*}M_{\chi_{D^{r}(a)}}H_{f}T$ is a compact positive operator on $L^{2}_{\phi}$. By Theorem \ref{thmkehe2}, and since we are dealing with the case of $p\geq 1$,
\begin{equation}
   \|B^{*}M_{\chi_{D^{r}(a)}}H_{f}T\|^{p}_{S_{p}}= \sup 
   \left\{\sum|\langle B^{*}M_{\chi_{D^{r}(a)}}H_{f}Te_{a},e_{a}\rangle : \{e_{a}\}_{a\in\Gamma}: \textrm{orthonormal}\right\}.
\end{equation}
So,
\begin{equation}
    \sum_{a\in\Gamma} |\langle B^{*}M_{\chi_{D^{r}(a)}}H_{f}Te_{a},e_{a}\rangle| \leq  \|B^{*}M_{\chi_{D^{r}(a)}}H_{f}T\|^{p}_{S_{p}}\leq C \|H_{f}\|^{p}_{S_{p}},
\end{equation}
as $\|B\|\leq 1$, $\|M_{\chi_{D^{r}(a)}}\|\leq 1$, and $\|T\|\leq C$. Recall that
\begin{equation}\label{S24}
    G_{2,r}(f)(a)=\inf\left\{\left(\frac{1}{|D^{r}(a)|}\int_{D^{r}(a)}|f-h|^{2}dA\right)^{1/2}: h\in H(D^{r}(a))\right\},
\end{equation}
and for $1\leq p<\infty$, $\|K_{z}\|_{p,\phi}\asymp e^{\phi(z)}\rho(z)^{2/p-2}$. Moreover, recalling Lemma 2.3 there exists $r_{0}>0$ such that for $w\in D^{r_{0}}(z)$,
\begin{equation}
    |K(w,z)|\asymp \frac{e^{\phi(w)+\phi(z)}}{\rho(z)^{2}}.
\end{equation}
Thus for $w\in D^{r_{0}}(z)$,
\begin{equation}\label{S26}
    |k_{p,z}(w)|e^{-\phi(w)}=\frac{|K(w,z)|}{\|K_{z}\|_{p,\phi}}e^{-\phi(w)} \asymp \frac{e^{\phi(w)+\phi(z)}e^{-\phi(w)}}{\rho(z)^{2}e^{\phi(z)}}\rho(z)^{-2/p+2}=\rho(z)^{-2/p}>0,
\end{equation}
and we can conclude that $\frac{P(fk_{2,z})}{k_{2,z}}\in H(D^{r}(z))$. Hence,
\begin{equation}
    G_{2,r}(f)(a)\leq \left[\frac{1}{|D^{r}(a)|}\int_{D^{r}(a)}|f-\frac{P(fk_{2,a})}{k_{2,a}}|^{2}dA\right]^{1/2}.
\end{equation}
Moreover,
\begin{align}
        \|\chi_{D^{r}(a)}H_{f}(k_{2,a})\|_{L^{2}_{\phi}}
        &= \left[\int_{D^{r}(a)}|fk_{2,a}-P(fk_{2,a})|^{2}e^{-2\phi}dA\right]^{1/2}\nonumber\\
        &= \left[\int_{D^{r}(a)}|f-\frac{P(fk_{2,a})}{k_{2,a}}|^{2}|k_{2,a}|^{2}e^{-2\phi}dA\right]^{1/2}\nonumber\\
        &\stackrel{(\ref{S26})}{\asymp} \left[\int_{D^{r}(a)}|f-\frac{P(fk_{2,a})}{k_{2,a}}|^{2}\rho(a)^{-2}dA\right]^{1/2} \nonumber\\
        &\asymp \left[\frac{1}{|D^{r}(a)|}\int_{D^{r}(a)}|f-\frac{P(fk_{2,a})}{k_{2,a}}|^{2}dA\right]^{1/2},
\end{align}
where in the last line we have used the equivalence $|D^{r}(z)|\asymp \rho(z)^{2}$. Hence,
\begin{equation}\label{bounlimcp}
    G_{2,r}(f)(a)\lesssim \|\chi_{D^{r}(a)}H_{f}(k_{2,a})\|_{L^{2}_{\phi}},
\end{equation}
and therefore,
\begin{align}
        \sum_{a\in\Gamma}G_{2,r}(f)(a)^{p}&\lesssim \sum_{a\in\Gamma} \|\chi_{D^{r}(a)}H_{f}(k_{2,a})\|_{L^{2}_{\phi}}^{p}\nonumber\\
        &=\sum_{a\in\Gamma} |\langle B^{*}M_{\chi_{D^{r}(a)}}H_{f}Te_{a},e_{a}\rangle|^{p} \leq C \|H_{f}\|^{p}_{S_{p}}.
\end{align}
Now note that 
\begin{align}\label{S31}         \|f\|^{p}_{\operatorname{IDA}^{p,2,-2/p}_{r}}
    &=
   \int_{\mathbb{C}}\rho^{-2}G_{2,r}(f)^{p}dA
    \nonumber\\
        &\leq \sum_{a\in\Gamma}\int_{D^{r}(a)}\rho(z)^{-2}G_{2,r}(f)(z)^{p}dA(z)
       \nonumber \\
       &\leq \sum_{a\in\Gamma}\sup_{z\in D^{r}(a)}\rho(z)^{-2}G_{2,r}(f)(z)^{p}|D^{r}(a)|\nonumber\\
       &=C \sum_{a\in\Gamma}\rho(a)^{-2}G_{2,r}(f)(a)^{p}\rho(a)^{2}\nonumber\\
       &= C \sum_{a\in\Gamma}G_{2,r}(f)(a)^{p}\nonumber\\
      &\leq C \|H_{f}\|_{S_{p}}^{p}.
\end{align}
Now since if Theorem \ref{Thm1.1} holds for some $r>0$, it holds for any $r$, we are done with the proof for $p\geq 1$.

Now we finish the proof of Theorem \ref{Thm1.2} by showing that $(1)\implies (2)$ for $0<p<1$. Since $H_{f}\in S_{p}(F^{2}_{\phi},L^{2}_{\phi})$, it is in particular bounded. For $a\in \Gamma$ set
 \begin{equation}  
g_{a} = 
     \begin{cases}
       \frac{\chi_{D^{r}(a)}fk_{2,a}-P_{a,r}(fk_{2,a})}{\|\chi_{D^{r}(a)}fk_{2,a}-P_{a,r}(fk_{2,a})\|} &\quad\text{if } \|\chi_{D^{r}(a)}fk_{2,a}-P_{a,r}(fk_{2,a})\|\neq 0,\\
       0 &\quad\text{if } \|\chi_{D^{r}(a)}fk_{2,a}-P_{a,r}(fk_{2,a})\|=0.\\ 
     \end{cases}
     \end{equation}
Then similar as before, $\|g_{a}\|\leq 1$, and $\langle g_{a},g_{b}\rangle=0$ for $a\neq b$. Let $J$ be any finite subcollection of $\Gamma$, and $\{e_{a}\}_{a\in J}$ be an orthonormal set of $L^{2}_{\phi}$. Define
\begin{equation}
    A=\sum_{a\in J}e_{a}\otimes g_{a}:L^{2}_{\phi}\to L^{2}_{\phi}.
\end{equation}
Then $A$ is of finite rank and $\|A\|\leq 1$. Similarly define
\begin{equation}
    T=\sum_{a\in J}k_{2,a}\otimes e_{a}:L^{2}_{\phi}\to F^{2}_{\phi}.
\end{equation}
Then as before, since $\Gamma$ is an $r$-lattice and thus separated, there is a constant $C$ such that $\|T\|\leq C$. Then,
\begin{equation}
    AH_{f}T=\sum_{a,\tau\in J}\langle H_{f}k_{2,\tau},g_{a}\rangle e_{a}\otimes e_{\tau}=Y+Z,
\end{equation}
where
\begin{equation}
    Y=\sum_{a\in J}\langle H_{f}k_{2,a},g_{a}\rangle e_{a}\otimes e_{a}, \quad\textrm{, }\quad Z=\sum_{a,\tau\in J, a\neq\tau}\langle H_{f}k_{2,\tau},g_{a}\rangle e_{a}\otimes e_{\tau}.
\end{equation}
Note that
\begin{align}
        \langle H_{f}k_{2,a},g_{a}\rangle_{2,\phi}&=\langle fk_{2,a}-P(fk_{2,a}),g_{a}\rangle_{2,\phi}=\langle \chi_{D^{r}(a)}fk_{2,a}-P_{a,r}(fk_{2,a}),g_{a}\rangle_{2,\phi}\nonumber\\
        &= \|\chi_{D^{r}(a)}fk_{2,a}-P_{a,r}(fk_{2,a})\|_{2,\phi}\nonumber\\
        &=\left[
        \int_{\mathbb{C}}|\chi_{D^{r}(a)}fk_{2,a}-P_{a,r}(fk_{2,a})|^{2} e^{-2\phi}dA
        \right]^{1/2}\nonumber\\
        &=\left[
        \int_{D^{r}(a)}|fk_{2,a}-P_{a,r}(fk_{2,a})|^{2} e^{-2\phi}dA
        \right]^{1/2}\nonumber\\
        &=\left[
        \int_{D^{r}(a)}|f-\frac{P_{a,r}(fk_{2,a})}{k_{2,a}}|^{2}|k_{2,a}|^{2} e^{-2\phi}dA
        \right]^{1/2}\nonumber\\
        &\asymp\left[ \frac{1}{|D^{r}(a)|}
        \int_{D^{r}(a)}|f-\frac{P_{a,r}(fk_{2,a})}{k_{2,a}}|^{2}dA
        \right]^{1/2}\nonumber\\
        &\geq G_{2,r}(f)(a).
\end{align}
where in the line before the last line we have used (\ref{S26}) and $|D^{r}(a)|\asymp \rho(a)^{2}$. Thus,
\begin{equation}
    \langle H_{f}k_{2,a},g_{a}\rangle_{2,\phi}\geq C G_{2,r}(f)(a).
\end{equation}
Therefore, there exists some $N$, independent of $f$ and $J$ such that
\begin{equation}
    \|Y\|^{p}_{S_{p}}=\sum_{a\in J} \langle H_{f}k_{2,a},g_{a}\rangle_{2,\phi}^{p}\geq N\sum_{a\in J} G_{2,r}(f)(a)^{p}.
\end{equation}
On the other hand for $0<p<1$,
\begin{equation}\label{S41}
    \|Z\|^{p}_{S_{p}}\leq \sum_{a,\tau\in J, a\neq \tau} \langle H_{f}k_{2,\tau},g_{a}\rangle_{2,\phi}^{p}.
\end{equation}
Let $Q_{a,r}:L^{2}(D^{r}(a),dA)\to A^{2}(D^{r}(a),dA)$ be the Bergman projection. Then $fk_{2,\tau}-P_{a,r}(fk_{2,\tau})$ and $P_{a,r}(fk_{2,\tau})-k_{2,\tau}Q_{a,r}f$ are orthogonal, and by Parseval's identity,
\begin{equation}\label{S42}
    \|fk_{2,\tau}-P_{a,r}(fk_{2,\tau})\|_{L^{2}(D^{r}(a),e^{-2\phi}dA)}\leq \|fk_{2,\tau}-k_{2,\tau}Q_{a,r}(f)\|_{L^{2}(D^{r}(a),e^{-2\phi}dA)}.
\end{equation}
Note that by Lemma 2.3, there exist $C,\epsilon>0$ such that
\begin{equation}\label{S43}
    |K(w,z)|\leq C\frac{e^{\phi(w)+\phi(z)}}{\rho(w)\rho(z)}e^{-\big(\frac{|z-w|}{\rho(z)}\big)^{\epsilon}}.
\end{equation}
Besides, by Lemma 6.8 in \cite{oliver2016toeplitz}, we can see that given $R>0$ and any finite sequence $\{a_{j}\}_{j=1}^{n}$ of different points in $\mathbb{C}$, it can be partitioned into subsequences such that any different points $a_{j}$ and $a_{k}$ in the same subsequence satisfy 
\begin{equation}\label{S44}
    |a_{j}-a_{k}|\geq R\min(\rho(a_{j}),\rho(a_{k})).
\end{equation}
So taking $J$ to be a finite collection of $\Gamma$, we can choose an appropriately large $R>0$ such that
\begin{equation}\label{S45}
    |a-b|\geq R\min(\rho(a),\rho(b)),\quad\textrm{when } a,b\in J,a\neq b.
\end{equation}
Putting everything together,
\begin{align}\label{S46}
        |\langle H_{f}k_{2,\tau},g_{a}\rangle|&= |\langle fk_{2,\tau}-P(fk_{2,\tau}),g_{a}\rangle|\nonumber\\
        &=|\langle fk_{2,\tau}-P(fk_{2,\tau}),\frac{\chi_{D^{r}(a)}fk_{2,a}-P_{a,r}(fk_{2,a})}{\|\chi_{D^{r}(a)}fk_{2,a}-P_{a,r}(fk_{2,a})\|}\rangle|\nonumber\\
        &=\frac{|\langle \chi_{D^{r}(a)}fk_{2,\tau}-P_{a,r}(fk_{2,\tau}),\chi_{D^{r}(a)}fk_{2,a}-P_{a,r}(fk_{2,a})\rangle|}{\|\chi_{D^{r}(a)}fk_{2,a}-P_{a,r}(fk_{2,a})\|} \nonumber\\
        &\leq \|fk_{2,\tau}-P_{a,r}(fk_{2,\tau})
        \|_{L^{2}(D^{r}(a),e^{-2\phi}dA)}\nonumber\\
        &\stackrel{(\ref{S42})}{\leq} \|fk_{2,\tau}-k_{2,\tau}Q_{a,r}(f)\|_{L^{2}(D^{r}(a),e^{-2\phi}dA)}\nonumber\\
        &\leq \sup_{\xi\in D^{r}(a)} |k_{2,\tau}(\xi)e^{-\phi}| \|f-Q_{a,r}(f)\|_{L^{2}(D^{r}(a),dA)}\nonumber\\
&\stackrel{(\ref{S43})}{\leq}
\sup_{\xi\in D^{r}(a)}\frac{C}{\rho(\xi)}e^{-\big(\frac{|\tau-\xi|}{\rho(\tau)} \big)^{\epsilon}}\|f-Q_{a,r}(f)\|_{L^{2}(D^{r}(a),dA)}\nonumber\\
&\simeq 
\frac{C}{\rho(a)}e^{-\big(\frac{|\tau-a|}{\rho(\tau)} \big)^{\epsilon}}\|f-Q_{a,r}(f)\|_{L^{2}(D^{r}(a),dA)}\nonumber\\
        &\simeq\frac{C}{|D^{r}(a)|^{1/2}}
        \left[
        \int_{D^{r}(a)}|f-Q_{a,r}(f)|^{2}dA
        \right]^{1/2}e^{-\big(\frac{|\tau-a|}{\rho(\tau)} \big)^{\epsilon}}\nonumber\\
        &=CG_{2,r}(f)(a)e^{-\big(\frac{|\tau-a|}{\rho(\tau)} \big)^{\epsilon}},
\end{align}
where in the last line we used the basic properties of Hilbert spaces. Therefore,
\begin{align}\label{S47}
        \|Z\|^{p}_{S_{p}}&\stackrel{(\ref{S41})}{\leq} \sum_{a,\tau\in J, a\neq \tau}G_{2,r}(f)(a)^{p} e^{-\big(\frac{|\tau-a|}{\rho(\tau)} \big)^{p\epsilon}}\nonumber\\
        &\stackrel{(\ref{S44})}{\leq}
        \sum_{a\in J}G_{2,r}(f)(a)^{p} \sum_{a,\tau\in J, a\neq \tau}e^{-\big(\frac{R\min(\rho(a),\rho(\tau))}{\rho(\tau)} \big)^{p\epsilon}}\nonumber\\
        &\simeq
        \sum_{a\in J}G_{2,r}(f)(a)^{p}e^{-R^{p\epsilon}}.
\end{align}
Now we can pick some $R$ large enough such that
\begin{equation}
    \|Z\|^{p}_{S_{p}}\leq \frac{N}{4}\sum_{a\in J}G_{2,r}(f)(a)^{p}.
\end{equation}
Using
\begin{equation}
    \|Y\|^{p}_{S_{p}} \leq 2 \|AH_{f}T\|^{p}_{S_{p}}+2\|Z\|^{p}_{S_{p}},
\end{equation}
we have
\begin{equation}
    N\sum_{a\in J}G_{2,r}(f)(a)^{p}\leq  2 \|AH_{f}T\|^{p}_{S_{p}}+\frac{N}{2} \sum_{a\in J}G_{2,r}(f)(a)^{p},
\end{equation}
and since $J$ is finite,
\begin{align}
         N\sum_{a\in J}G_{2,r}(f)(a)^{p}&\leq  2 \|AH_{f}T\|^{p}_{S_{p}}\nonumber\\
         &\leq 4 \|A\|^{p}_{L^{2}_{\phi}\to L^{2}_{\phi}} \|H_{f}\|^{p}_{S_{p}} \|T\|^{p}_{L^{2}_{\phi}\to L^{2}_{\phi}}\nonumber\\
         &\leq C\|H_{f}\|^{p}_{S_{p}}.
\end{align}
Since $C$ is independent of $f$ and $J$, 
\begin{equation}
   \sum_{a\in \Gamma}G_{2,r}(f)(a)^{p}\leq 
   C\|H_{f}\|^{p}_{S_{p}}. 
\end{equation}
The remaining of the proof is similar to (\ref{S31}) and we can conclude that for $0<p<1$, 
\begin{equation}
    \|f\|_{\operatorname{IDA}^{p,2,-2/p}_{r}}\leq C \|H_{f}\|^{p}_{S_{p}}.
\end{equation}
\end{proof}

\section{Simultaneous membership of $H_{f}$ and $H_{\bar{f}}$ in $S_{p}$}\label{simul_section}

In this section, we first define the space of functions of integral mean oscillation $\operatorname{IMO}$ and prove some of its basic properties. In particular, we prove that $H_f$ and $H_{\bar{f}}$ are simultaneously in $S_p(F^2_\phi, L^2_\phi$) with $0<p<\infty$ if and only if the symbol $f$ satisfies a suitable $\operatorname{IMO}$ condition (see Theorem~\ref{Thm5.4}).

\begin{lem}\label{Lemsimul1}
 Let $0<p<\infty$ and $r>0$. Then for $f\in L^{2}_{loc}$, $f\in \operatorname{IMO}^{p,2,\alpha}_{r}$ if and only if there exists a continuous function $c(z)$ on $\mathbb{C}$ such that
 \begin{equation}\label{M3}
     \rho^{\alpha}\left(
\frac{1}{|D^{r}(z)|}\int_{D^{r}(z)}|f(w)-c(z)|^{2}dA(w)
     \right)^{1/2}\in L^{p}
 \end{equation}
\end{lem}
\begin{proof}
    This proof is similar to the proof of Proposition 2.4 in \cite{huWang2018hankel}. We can similarly extend the proposition to the case $0<p<1$, and the doubling weights by introducing $\rho$ as the following. First note that if $f\in \operatorname{IMO}^{p,2,\alpha}_{r}$, then (\ref{M3}) holds with $c(z)=\hat{f}_{r}(z)$ which is continuous for $z\in\mathbb{C}$. Conversely, assume that (\ref{M3}) holds. By Minkowski inequality,
    \begin{align}\label{M4}
        \rho^{\alpha}(z)MO_{2,r}(f)(z)&\leq
        \rho^{\alpha} \big(  
        \frac{1}{|D^{r}(z)|}\int_{D^{r}(z)}|f-c(z)|^{2}dA\big)^{1/2}
        +
    \rho^{\alpha}
        |\hat{f}_{r}(z)-c(z)|. 
    \end{align}
    By H\"{o}lder's inequality,
    \begin{align}\label{M5}
          \rho^{\alpha}
        |\hat{f}_{r}(z)-c(z)|&\leq\rho^{\alpha}\big( 
        \frac{1}{|D^{r}(z)|}\int_{D^{r}(z)}|f-c(z)|^{2}dA \big)^{1/2}\in L^{p}\quad\textrm{by (\ref{M3})}.
    \end{align}
    Hence, using (\ref{M4}) and (\ref{M5}) we can see that $f\in \operatorname{IMO}^{p,2,\alpha}_{r}$.
\end{proof}

\begin{prop}
    Let $0<p\leq \infty$, $r>0$, and $f\in L^{2}_{loc}$. If for each $z\in\mathbb{C}$, there exist $h_{1},h_{2}\in H(D^{r}(z))$ such that
    \begin{align}\label{M6}
            &\rho^{\alpha}(z)\big(
\frac{1}{|D^{r}(z)|}\int_{D^{r}(z)}|f-h_{1}|^{2}dA
            \big)^{1/2}\in L^{p},\nonumber\\
            &\textrm{and}\nonumber\\
            &\rho^{\alpha}(z)\big(
\frac{1}{|D^{r}(z)|}\int_{D^{r}(z)}|\bar{f}-h_{2}|^{2}dA
            \big)^{1/2}\in L^{p},
    \end{align}
    then $f\in \operatorname{IMO}^{p,2,\alpha}_{r}$.
\end{prop}
\begin{proof}
    The proof is a more detailed version of the proof of Proposition 2.5 in \cite{huWang2018hankel}, extended to the case of doubling Fock spaces. For $f\in L^{2}_{loc}$, recall that
    \begin{equation}\label{M7}
        \big(
        \widehat{|f|^{2}}_{r}(z)
        \big)^{1/2}= \big(\frac{1}{|D^{r}(z)|}\int_{D^{r}(z)}|f|^{2}dA\big)^{1/2}.
    \end{equation}
By the triangle inequality and using (\ref{M6}),
\begin{align}\label{M8}
        \rho^{\alpha}\big( 
        \reallywidehat{|\frac{f+\bar{f}}{2}-\frac{h_{1}+h_{2}}{2}|^{2}}_{r}(z)
        \big)^{1/2}
       &\leq
        \rho^{\alpha}\big( 
        \reallywidehat{|\frac{f-h_{1}}{2}|^{2}}_{r}(z)
        \big)^{1/2}
        + \rho^{\alpha}\big( 
        \reallywidehat{|\frac{\Bar{f}-h_{2}}{2}|^{2}}_{r}(z)
        \big)^{1/2}\in L^{p}.
\end{align}
Since $f+\bar{f}$ and $\rho^{\alpha}$ are real-valued, we can conclude that
\begin{equation}\label{M9}
    \rho^{\alpha}\big( 
\reallywidehat{|\operatorname{Im}\frac{h_{1}+h_{2}}{2}|^{2}}_{r}(z)
        \big)^{1/2}\in L^{p}.
\end{equation}
As in the proof of the Proposition 2.5 in \cite{huWang2018hankel}, we know that if $v:D^{r}(z)\to \mathbb{R}$ is harmonic, there exists a harmonic function $u$ such that $u+iv\in H(D^{r}(z))$ and
\begin{equation}\label{M10}
    \|u-u(z)\|_{L^{q}(D^{r}(z),dA)}\leq C \|v\|_{L^{q}(D^{r}(z),dA)},
\end{equation}
for all $0<q<\infty$. 

Taking $q=2$ in (\ref{M10}), and since $h_{1}+h_{2}\in H(D^{r}(z))$, 
\begin{equation}\label{M11}
\big(
\reallywidehat{|\operatorname{Re}\frac{h_{1}+h_{2}}{2}-\operatorname{Re}\frac{h_{1}+h_{2}}{2}(z)|^{2}}_{r}(z)
\big)^{1/2}
   \leq C \big( 
\reallywidehat{|\operatorname{Im}\frac{h_{1}+h_{2}}{2}|^{2}}_{r}(z)
        \big)^{1/2}.
\end{equation}
Thus,
\begin{align}\label{M12}
     \rho^{\alpha}\big(
\reallywidehat{|\frac{f+\Bar{f}}{2}-\operatorname{Re}\frac{h_{1}+h_{2}}{2}(z)|^{2}}_{r}(z)
\big)^{1/2}
&\leq
\rho^{\alpha}\big(
\reallywidehat{|\frac{f+\bar{f}}{2}-\operatorname{Re}\frac{h_{1}+h_{2}}{2}|^{2}}_{r}(z)
\big)^{1/2}\nonumber\\
&\qquad\qquad+
\rho^{\alpha}\big(
\reallywidehat{|\operatorname{Re}\frac{h_{1}+h_{2}}{2}-\operatorname{Re}\frac{h_{1}+h_{2}}{2}(z)|^{2}}_{r}(z)
\big)^{1/2}\nonumber\\
&\leq
\rho^{\alpha}\big(
\reallywidehat{|\frac{f+\Bar{f}}{2}-\frac{h_{1}+h_{2}}{2}|^{2}}_{r}(z)
\big)^{1/2} \nonumber\\
&\qquad\qquad+
C\rho^{\alpha}\big( 
\reallywidehat{|\operatorname{Im}\frac{h_{1}+h_{2}}{2}|^{2}}_{r}(z)
        \big)^{1/2}\in L^{p},
\end{align}
where the first term in the last line is in $L^{p}$ by (\ref{M8}), and the second term is in $L^{p}$ by (\ref{M9}). Hence,
\begin{equation}\label{M13}
    \rho^{\alpha}\big(
\reallywidehat{|\frac{f+\Bar{f}}{2}-\operatorname{Re}\frac{h_{1}+h_{2}}{2}(z)|^{2}}_{r}(z)
\big)^{1/2}\in L^{p}.
\end{equation}
Similar to (\ref{M8}), (\ref{M9}), and (\ref{M10}), and applying (\ref{M6}), we have
\begin{align}\label{M14}
        \rho^{\alpha}\big( 
        \reallywidehat{|\frac{f-\bar{f}}{2}-\frac{h_{1}-h_{2}}{2}|^{2}}_{r}(z)
        \big)^{1/2}
        &\leq
        \rho^{\alpha}\big( 
        \reallywidehat{|\frac{f-h_{1}}{2}|^{2}}_{r}(z)
        \big)^{1/2}
        + \rho^{\alpha}\big( 
        \reallywidehat{|\frac{\Bar{f}-h_{2}}{2}|^{2}}_{r}(z)
        \big)^{1/2}\in L^{p}
\end{align} 
Since $\frac{f-\Bar{f}}{2}$ is completely imaginary, we can conclude that
\begin{equation}\label{M15}
    \rho^{\alpha}\big( 
\reallywidehat{|\operatorname{Re}\frac{h_{1}-h_{2}}{2}|^{2}}_{r}(z)
        \big)^{1/2}\in L^{p}.
\end{equation}
We can exchange $u$ and $v$ in (\ref{M10}), and therefore,
\begin{equation}\label{M16}
\big(
\reallywidehat{|\operatorname{Im}\frac{h_{1}-h_{2}}{2}-\operatorname{Im}\frac{h_{1}-h_{2}}{2}(z)|^{2}}_{r}(z)
\big)^{1/2}
   \leq C \big( 
\reallywidehat{|\operatorname{Re}\frac{h_{1}-h_{2}}{2}|^{2}}_{r}(z)
        \big)^{1/2}.
\end{equation}
Thus by (\ref{M14}) and (\ref{M15}),
\begin{align}\label{M17}
     \rho^{\alpha}\big(
\reallywidehat{|\frac{f-\Bar{f}}{2}-\operatorname{Im}\frac{h_{1}-h_{2}}{2}(z)|^{2}}_{r}(z)
\big)^{1/2}& 
\leq
\rho^{\alpha}\big(
\reallywidehat{|\frac{f-\bar{f}}{2}-\operatorname{Im}\frac{h_{1}-h_{2}}{2}|^{2}}_{r}(z)
\big)^{1/2}\nonumber\\
&\qquad\qquad+
\rho^{\alpha}\big(
\reallywidehat{|\operatorname{Im}\frac{h_{1}-h_{2}}{2}-\operatorname{Im}\frac{h_{1}-h_{2}}{2}(z)|^{2}}_{r}(z)
\big)^{1/2}\nonumber\\
&\leq
\rho^{\alpha}\big(
\reallywidehat{|\frac{f-\Bar{f}}{2}-\frac{h_{1}-h_{2}}{2}|^{2}}_{r}(z)
\big)^{1/2}\nonumber\\
&\qquad\qquad+
C\rho^{\alpha}\big( 
\reallywidehat{|\operatorname{Re}\frac{h_{1}-h_{2}}{2}|^{2}}_{r}(z)
        \big)^{1/2}\in L^{p}.
\end{align}
Hence, analogous to (\ref{M13}),
\begin{equation}\label{M18}
    \rho^{\alpha}\big( \reallywidehat{|\frac{f-\Bar{f}}{2}-\operatorname{Im}\frac{h_{1}-h_{2}}{2}(z)|^{2}}_{r}(z) \big)^{1/2}\in L^{p}.
\end{equation}
Choose $c(z)=\operatorname{Re}\frac{h_{1}+h_{2}}{2}(z)+i\operatorname{Im}\frac{h_{1}-h_{2}}{2}(z)$. Then by (\ref{M13}) and (\ref{M18}),
\begin{equation}\label{M19}
\rho^{\alpha}\big(\reallywidehat{|f-c(z)|^{2}}_{r}(z) \big)^{1/2}\in L^{p},
\end{equation}
which is equivalent to
\begin{equation}\label{M20}
    \rho^{\alpha}\big(
    \frac{1}{|D^{r}(z)|}\int_{D^{r}(z)} |f-c(z)|^{2}dA    \big)^{1/2}\in L^{p}.
\end{equation}
Thus by Lemma \ref{Lemsimul1} we can conclude that $f\in \operatorname{IMO}^{p,2,\alpha}_{r}$.  
\end{proof}
\begin{lem}\label{lem4.3}
   Let $0<p\leq \infty$. Then for $f\in L^{2}_{loc}$, $f\in \operatorname{IDA}^{p,2,\alpha}_{r}$ and $\bar{f}\in \operatorname{IDA}^{p,2,\alpha}_{r}$ if and only if $f\in \operatorname{IMO}^{p,2,\alpha}_{r}$. 
\end{lem}
\begin{proof}
    First, we show that
    \begin{equation}\label{M21}
\|f\|_{\operatorname{IMO}^{p,2,\alpha}_{r}}=\|\rho^{\alpha}MO_{2,r}(f)\|_{L^{p}}\lesssim \|f\|_{\operatorname{IDA}^{p,2,\alpha}_{r}}+\|\bar{f}\|_{\operatorname{IDA}^{p,2,\alpha}_{r}}.
    \end{equation}
Note that by Lemma \ref{lemIDA1}, there exists $h_{1},h_{2}\in H(D^{r}(z))$ such that
\begin{equation}\label{M22}
        G_{2,r}(f)(z)=\big(\reallywidehat{|f-h_{1}|^{2}}_{r}(z) \big)^{1/2},
        \quad
        \textrm{and }\quad
        G_{2,r}(\bar{f})(z)=\big(\reallywidehat{|\bar{f}-h_{2}|^{2}}_{r}(z) \big)^{1/2}.
\end{equation}
Taking $c(z)$ as in the proof of the previous lemma, and using (\ref{M12}), (\ref{M17}), (\ref{M8}), and (\ref{M14}),
\begin{align}\label{M23}   \rho^{\alpha}\big(\reallywidehat{|f-c(z)|^{2}}_{r}(z) \big)^{1/2}&
        =
\rho^{\alpha}\big(\reallywidehat{|\frac{f+\Bar{f}}{2}-\operatorname{Re}\frac{h_{1}+h_{2}}{2}(z)+\frac{f-\bar{f}}{2}-i\operatorname{Im}\frac{h_{1}-h_{2}}{2}(z)|^{2}}_{r}(z) \big)^{\frac12}\nonumber\\
        &\leq C\rho^{\alpha}\big( G_{2,r}(f)(z)+G_{2,r}(\Bar{f})(z) \big)\nonumber\\
        & +C\rho^{\alpha}\big\{
        \big( 
\reallywidehat{|\operatorname{Im}\frac{h_{1}+h_{2}}{2}|^{2}}_{r}(z)
        \big)^{1/2}+ \big( 
\reallywidehat{|\operatorname{Re}\frac{h_{1}-h_{2}}{2}|^{2}}_{r}(z)
        \big)^{1/2}
        \big\}.
\end{align}
Note that since $L^{2}$ is a Hilbert space, we can set $h_{1}=Q_{z,r}(f)$ and $h_{2}=Q_{z,r}(\Bar{f})$. Then the linearity of the Bergman projection $Q_{z,r}:L^{2}(D^{r}(z),dA)\to A^{2}(D^{r}(z),dA)$ implies that the last two terms are zero. Thus,
\begin{equation}\label{M24}
  \rho^{\alpha}\big(\reallywidehat{|f-c(z)|^{2}}_{r}(z) \big)^{1/2} \leq  C\rho^{\alpha}\big( G_{2,r}(f)(z)+G_{2,r}(\Bar{f})(z) \big).  
\end{equation}
Hence,
\begin{align}\label{M25}
        \rho^{\alpha}MO_{2,r}(f)(z)
        \leq
        \rho^{\alpha}\big(
        \frac{1}{|D^{r}(z)|}\int_{D^{r}(z)}|f-c(z)|^{2}dA
        \big)^{1/2}+
        \rho^{\alpha}|\hat{f}_{r}(z)-c(z)|.
\end{align}
By H\"{o}lder's inequality,
\begin{align}\label{M26}
      |\hat{f}_{r}(z)-c(z)|
      &\leq
      \big(
\frac{1}{|D^{r}(z)|}\int_{D^{r}(z)}|f-c(z)|^{2}dA
      \big)^{1/2}.
\end{align}
Applying this to (\ref{M25}), and using (\ref{M24}), we get
\begin{align}\label{M27}
         \rho^{\alpha}MO_{2,r}(f)(z)
      \leq C\rho^{\alpha}\big\{ G_{2,r}(f)(z)+G_{2,r}(\bar{f})(z) \big\}.
\end{align}
Taking the $L^{p}$-norms of both sides we can conclude that for $0<p\leq \infty$,
\begin{equation}\label{M28}
\|f\|_{\operatorname{IMO}^{p,2,\alpha}_{r}}\lesssim \|f\|_{\operatorname{IDA}^{p,2,\alpha}_{r}}+\|\bar{f}\|_{\operatorname{IDA}^{p,2,\alpha}_{r}}.
    \end{equation}
    
For the inverse inequality, note that using the definition, it is immediate to see that $f\in \operatorname{IMO}^{p,2,\alpha}_{r}$ if and only if $\bar{f}\in \operatorname{IMO}^{p,2,\alpha}_{r}$. Moreover, $\hat{f}_{r}(z)$ is a constant, and therefore holomorphic. So by definition, $\|f\|_{\operatorname{IDA}^{p,2,\alpha}_{r}}\leq \|f\|_{\operatorname{IMO}^{p,2,\alpha}_{r}}$. Similarly, $\|\bar{f}\|_{\operatorname{IDA}^{p,2,\alpha}_{r}}\leq \|\bar{f}\|_{\operatorname{IMO}^{p,2,\alpha}_{r}}=\|f\|_{\operatorname{IMO}^{p,2,\alpha}_{r}}$, and we are done.
\end{proof}

We can now give the proof of Theorem~\ref{Thm5.4}, which shows that both $H_f$ and $H_{\bar f}$ are in $S_p$ if and only if $f\in \IMO_r^{p,2,-2/p}$, where $1<p<\infty$.

\begin{proof}[Proof of Theorem \ref{Thm5.4}]
 By Theorem \ref{Thm1.2}, $H_{f}\in S_{p}$ if and only if $f\in \IDA^{p,2,-2/p}_{r}$ for some (equivalent any) $r>0$. Similarly, $H_{\bar{f}}\in S_{p}$ if and only if $\bar{f}\in \IDA^{p,2,-2/p}_{r}$. An application of Lemma \ref{lem4.3} shows that this is equivalent to $f\in \operatorname{IMO}^{p,2,-2/p}_{r}$, for some (equivalent any) $r>0$. Further, the norm estimates in~\eqref{M29} follow from \eqref{Schatten norm} and \eqref{M21}.
\end{proof}

As mentioned in the introduction, we obtain the following result as a consequence of Theorem~\ref{Thm5.4}.

\begin{thm}\label{Schneider_thm}
Let $f$ be a non-constant entire function and $F^2_\phi$ be a doubling Fock space. Then $H_{\bar{f}}$ is not in $S_{2}(F^{2}_{\phi},L^{2}_{\phi})$. 
    \end{thm}
    
\begin{proof}
Since $f$ is holomorphic, $H_{f}=0$, and thus belongs to the Hilbert-Schmidt class. Applying Theorem \ref{Thm5.4}, it is enough to show that $f\notin \operatorname{IMO}^{2,2,-1}_{1}$. First note that $\Bar{f}$ is harmonic on $D^{1}(z)$ and by the mean-value property of harmonic functions,
\begin{equation*}
    \widehat{f_{1}}(z)=\frac{1}{|D^{1}(z)|}\int_{D^{1}(z)} fdA=f(z).
\end{equation*}
By the Cauchy estimate,
\begin{align*}\label{B17}
        MO_{2,1}(f)(z)&=\left( \frac{1}{|D^{1}(z)|}\int_{D^{1}(z)}|f(w)-f(z)|^{2} dA(w) \right)^{1/2}\nonumber\\
        &\geq C|\partial f(z)|\rho(z).
\end{align*}
Hence,
\begin{align*}
    \|f\|_{\operatorname{IMO}^{2,2,-1}_{1}} &=\int_{\mathbb{C}}\rho(z)^{-2}MO_{2,1}(f)(z)^{2}dA(z)\\
    &\geq C\int_{\mathbb{C}}\rho(z)^{-2}|\partial f(z)|^{2}\rho(z)^{2}dA(z).
\end{align*}
So, since $f$ is entire and non-constant, it follows that $f\notin \operatorname{IMO}^{2,2,-1}_{1}$, and thus $H_{\bar{f}}$ is not Hilbert-Schmidt.

\end{proof}

\section{Berger-Coburn phenomenon for doubling Fock spaces}

This section contains the proofs of Theorems~\ref{Thm1.3} and \ref{thm1.4}. We start with the proof of the Berger-Coburn phenomenon for Hilbert-Schmidt Hankel operators, that is, we show that for $f\in L^\infty$, $H_f$ is Hilbert-Schmidt if and only if $H_{\bar f}$ is Hilbert-Schmidt.

\begin{proof}[Proof of Theorem \ref{Thm1.3}]
   Let $H_{f}\in S_{2}$. By the assumption, $f\in L^{\infty}$, and in particular $f\in L^{2}_{loc}$. Then by Theorem \ref{Thm1.2}, $f\in\operatorname{IDA}^{2,2,-1}_{r}$ for some (equivalent any) $r>0$, and
   \begin{equation}\label{B2}
       \|f\|_{\operatorname{IDA}^{2,2,-1}_{r}} \simeq \|H_{f}\|_{S_{2}}<\infty.
   \end{equation}
Decompose $f=f_{1}+f_{2}$ as in (\ref{D1}). Thus $f_{1}\in \mathcal{C}^{2}(\mathbb{C})$ and 
\begin{equation}\label{B3}
        |\Bar{\partial} f_{1}|+(\widehat{|\Bar{\partial} f_{1}|^{2}}_{r})^{1/2}+ \rho^{-1}(\widehat{| f_{2}|^{2}}_{r})^{1/2} \in L^{2}.
    \end{equation}
Then the definition 
\begin{equation}\label{B4}
   \rho^{-1}(z)(\widehat{| f_{2}|^{2}}_{r}(z))^{1/2}=\rho^{-1}(z)\big( \frac{1}{|D^{r}(z)|}\int_{D^{r}(z)}|f_{2}|^{2}dA \big)^{1/2} 
\end{equation}
implies that
\begin{equation}\label{B5}
    \rho^{-1}(\widehat{| f_{2}|^{2}}_{r})^{1/2}=
    \rho^{-1}(\widehat{| \bar{f_{2}}|^{2}}_{r})^{1/2}\in L^{2}.
\end{equation}
 By (\ref{D2}) and (\ref{Schatten norm}), $H_{\bar{f_{2}}}\in S_{2}$. Indeed,
 \begin{align}\label{B6}
     \|H_{\Bar{f_{2}}}\|_{S_{2}}
     &\stackrel{(\ref{Schatten norm})}{=}
\|\bar{f_{2}}\|_{\operatorname{IDA}^{2,2,-1}_{r}}
     \stackrel{(\ref{D2})}{\lesssim}
     \|\rho^{-1}(\widehat{| \bar{f_{2}}|^{2}}_{r})^{1/2}\|_{L^{2}}\nonumber\\
  &\stackrel{(\ref{B5})}{=}
     \|\rho^{-1}(\widehat{| f_{2}|^{2}}_{r})^{1/2}\|_{L^{2}}
     \stackrel{(\ref{D2})}{\lesssim}
\|f\|_{\operatorname{IDA}^{2,2,-1}_{r}}.
 \end{align}
 
To show that $\|H_{\bar{f_{1}}}\|_{S_{2}}\lesssim \|f\|_{\operatorname{IDA}^{2,2,-1}_{r}}$, we need to follow a more complicated argument, inspired by the proof of Theorem 1.2 in \cite{hu2023corrigendum}. Let $\{a_{j}\}_{j=1}^{\infty}$ be a fixed $m_{1}r$-lattice for some $m_{1}\in (0,1)$ and $r>0$. Choose a partition of unity $\{\psi_{j}\}_{j=1}^{\infty}$ subordinate to $\{D^{m_{1}r}(a_{j})\}$ as in (\ref{D11}). By Lemma \ref{lemIDA1} there exists $h_{j}\in H(D^{r}(a_{j}))$ such that
\begin{equation}\label{R3}
    \big(\reallywidehat{|f-h_{j}|^{2}}_{r}(a_{j})\big)^{1/2}=G_{2,r}(f)(a_{j}),\quad \textrm{and } \sup_{z\in D^{m_{1}r}(a_{j})} |h_{j}(z)|\lesssim \|f\|_{L^{\infty}}.
\end{equation}
Now we get back to the decomposition $f=f_{1}+f_{2}$ as in (\ref{D1}) with $f_{1}=\sum_{j=1}^{\infty}h_{j}\psi_{j}$. Without loss of generality we can assume $\psi_{j}=\bar{\psi}_{j}$ for all $j\geq 1$. Since we assumed that $f$ is bounded, $f_{1}\in L^{\infty}$ and moreover
\begin{equation}\label{R4}
\Bar{\partial}\Bar{f_{1}}=
    \sum_{j=1}^{\infty}\Bar{h}_{j}\Bar{\partial}\psi_{j}+ \sum_{j=1}^{\infty}\psi_{j}\Bar{\partial}\bar{h}_{j}=F+H,
\end{equation}
for $F=\sum_{j=1}^{\infty}\Bar{h}_{j}\Bar{\partial}\psi_{j}$ and $H=\sum_{j=1}^{\infty}\psi_{j}\Bar{\partial}\bar{h}_{j}$. Similar to (\ref{D15}) one has
\begin{align}\label{R5}
    |F(z)|&=\rho^{-1}(z)\rho(z)|\sum_{j=1}^{\infty}\bar{h}_{j}\Bar{\partial}\psi_{j}|=
    \rho^{-1}(z)\rho(z)|\sum_{j=1}^{\infty}\bar{h}_{j}\Bar{\partial}\psi_{j}-\sum_{j=1}^{\infty}\bar{h}_{1}\Bar{\partial}\psi_{j}|\nonumber\\
    & \leq \rho^{-1}(z)\rho(z)\sum_{j=1}^{\infty}|\bar{h}_{j}(z)-\bar{h}_{1}(z)||\Bar{\partial}\psi_{j}(z)|
    \leq C \rho^{-1}(z)G_{2,r}(f)(z).
\end{align}
Besides,
\begin{equation}
    \|H\|_{L^{2}}\leq \|\Bar{\partial}\bar{f}_{1}\|_{L^{2}}+\|F\|_{L^{2}}.
\end{equation}
By (\ref{R5}),
\begin{equation}
    \|F\|_{L^{2}}\leq \|f\|_{\operatorname{IDA}^{2,2,-1}_{r}}.
    \end{equation}
Lemma 7.1 in \cite{Janihu2022schatten} implies that
\begin{equation}\label{R6}
    \|\Bar{\partial}\Bar{f}_{1}\|_{L^{2}}= \|\partial f_{1}\|_{L^{2}}\leq
    C \|\Bar{\partial}f_{1}\|_{L^{2}}\leq C\|f\|_{\operatorname{IDA}^{2,2,-1}_{r}},
\end{equation}
where the last inequality is obtained by multiplying both sides of (\ref{D14}) with $\rho^{-1}$. Hence, we can conclude that
\begin{equation}\label{RR6}
    \|H\|_{L^{2}}\lesssim \|f\|_{\operatorname{IDA}^{2,2,-1}_{r}}.
\end{equation}
Note that for $m_{1},m_{2}\in (0,1)$,
\begin{align}\label{R7}
\|H_{\bar{f}_{1}}\|_{S_{2}}^{2}&\simeq \|\bar{f}_{1}\|_{\operatorname{IDA}^{2,2,-1}_{r}}^{2} \stackrel{(\ref{D1})}{\leq} C\int_{\mathbb{C}}\big[
(\reallywidehat{|\Bar{\partial}\bar{f}_{1}|^{2}}_{m_{}m_{2}r})^{1/2}
\big]^{2}dA \nonumber\\
& \lesssim \int_{\mathbb{C}} \big[ 
(\reallywidehat{|F|^{2}}_{m_{1}m_{2}r})^{1/2}
\big]^{2}dA+
\int_{\mathbb{C}} \big[ 
(\reallywidehat{|H|^{2}}_{m_{1}m_{2}r})^{1/2}
\big]^{2}dA,
\end{align}
where for the last inequality we used the equivalence $\rho(w)\simeq\rho(z)$ for $w\in D^{m_{1}m_{2}r}(z)$ and (\ref{R4}).
Note that using (\ref{R5}) one has
\begin{equation}\label{R8}
    \int_{\mathbb{C}} \big[ 
(\reallywidehat{|F|^{2}}_{m_{1}m_{2}r})^{1/2}
\big]^{2}dA \lesssim \|f\|_{\operatorname{IDA}^{2,2,-1}_{r}}^{2},
\end{equation}
and thus we are left to compute $\int_{\mathbb{C}} \big[ 
(\reallywidehat{|H|^{2}}_{m_{1}m_{2}r})^{1/2}
\big]^{2}dA$.
Let $z\in D^{r}(a_{j})\cap D^{r}(a_{k})$. Since $|\bar{\partial}(\Bar{h}_{k}-\Bar{h}_{j})|=|\partial(h_{k}-h_{j})|$, 
 applying the Cauchy estimate for the boundary of the disk $D^{m_{1}m_{2}r}(z)$ of radius $m_{1}m_{2}r\rho(z)$ and H\"{o}lder's inequality, we obtain the following.
\begin{equation}
    |\bar{\partial}(\Bar{h}_{k}(z)-\Bar{h}_{j}(z))|\leq \frac{C}{\rho(z)}\big\{\int_{D^{m_{1}m_{2}r}(z)}|\Bar{h}_{k}(w)-\Bar{h}_{j}(w)|^{2}dA\big\}^{1/2}.
\end{equation}
Using $|\Bar{h}_{k}-\Bar{h}_{j}|^{2}=|(f-\Bar{h}_{k})-(f-\Bar{h}_{j})|^{2}\leq |f-\Bar{h}_{k}|^{2}+|f-\Bar{h}_{j}|^{2}$, and the fact that $h_{k}$ and $h_{j}$ are holomorphic, we get
\begin{align}\label{R9}
 |\bar{\partial}(\Bar{h}_{k}(z)-\Bar{h}_{j}(z))|&\leq \frac{C}{\rho(z)}\big(
 G_{2,m_{1}m_{2}r}(f)(a_{k})+G_{2,m_{1}m_{2}r}(f)(a_{j})
 \big)\nonumber\\
 &\leq \frac{C}{\rho(z)}G_{2,R}(f)(z),
\end{align}
for some $R>m_{1}m_{2}r$. Recalling $H$ as in (\ref{R4}), 
\begin{equation}
    H+\sum_{j=1}^{\infty}\psi_{j}\bar{\partial}(\bar{h}_{k}-\bar{h}_{j})=H+ \sum_{j=1}^{\infty}\psi_{j}\bar{\partial}\bar{h}_{k}-H.
\end{equation}
since $\{\psi_{j}\}_{j=1}^{\infty}$ is a partition of unity and therefore $\sum_{j=1}^{\infty}\psi_{j}=1$, 
\begin{equation}
    \Bar{\partial}\bar{h}_{k}=
    \sum_{j=1}^{\infty}\psi_{j}\bar{\partial}(\bar{h}_{k}-\bar{h}_{j})+H.
\end{equation}
Hence,
\begin{align}\label{R10}
    |\Bar{\partial}\bar{h}_{k}(z)|^{2}&\lesssim \big|
    \sum_{j=1}^{\infty}\psi_{j}(z)\bar{\partial}(\bar{h}_{k}(z)-\bar{h}_{j}(z))
    \big|^{2}+\big|H(z)\big|^{2}\nonumber\\
    & \lesssim \sum_{j\in D^{m_{1}r}(a_{j})}\psi_{j}(z) |\bar{\partial}(\bar{h}_{k}(z)-\bar{h}_{j}(z))
    |^{2}+|H(z)|^{2}\nonumber\\
    & \lesssim \big(\rho^{-1}(z)G_{2,R}(f)(z)\big)^{2}+|H(z)|^{2},
\end{align}
where the last inequality follows from \eqref{R9}.
For $z\in D^{m_{1}r}(a_{k})$, notice that $D^{m_{1}m_{2}r}(z)\subset D^{m_{1}r}(a_{k})$ for some $m_{2}\in (0,1)$. Then by subharmonicity,
\begin{align}\label{R11}
    |\Bar{\partial}\bar{h}_{k}(z)|^{2}&\leq \frac{1}{|D^{m_{1}m_{2}r}(z)|}\int_{D^{m_{1}m_{2}r}(z)} |\Bar{\partial}\Bar{h}_{k}(w)|^{2}dA(w) \nonumber\\
    &\stackrel{(\ref{R10})}{\lesssim} 
    \frac{1}{|D^{m_{1}m_{2}r}(z)|}\int_{D^{m_{1}m_{2}r}(z)}\bigg[
    \big|\rho^{-1}(w)G_{2,R}(f)(w)\big|^{2}
    +\big|H(w)\big|^{2}
    \bigg]dA(w)\nonumber\\
    &\lesssim (\rho^{-1}(z))^{2}G_{2,\Tilde{R}}(f)(z)^{2}+\reallywidehat{|H|^{2}}_{m_{1}m_{2}r}(z),
\end{align}
for some $\Tilde{R}>R$.

Now for $z\in \mathbb{C}$, there exists $w'\in \overline{D^{m_{1}m_{2}r}(z)}$ such that
\begin{align}\label{R12}
    \big[(\reallywidehat{|H|^{2}}_{m_{1}m_{2}r}(z))^{1/2}
    \big]^{2}& \leq \max \{|H(w)|^{2}:w\in \overline{D^{m_{1}m_{2}r}(z)}\} \nonumber\\
    &=\big| \sum_{k=1}^{\infty}\psi_{k}(w')\bar{\partial}\bar{h}_{k}(w') \big|^{2},
\end{align}
where the first inequality comes from integration on a bounded domain. Note that $G_{2,\Tilde{R}}(f)(w')^{2}\lesssim G_{2,s}(f)(z)^{2}$ for some $s>\Tilde{R}$, and 
\begin{equation}\label{R13}
    \big[(\reallywidehat{|H|^{2}}_{m_{1}m_{2}r}(w'))^{1/2}\big]^{2}
    \leq 
\big[(\reallywidehat{|H|^{2}}_{m_{1}r}(z))^{1/2}\big]^{2},
\end{equation}
and we can conclude that
\begin{align}\label{R14}
\big[(\reallywidehat{|H|^{2}}_{m_{1}m_{2}r}(z))^{1/2}\big]^{2}& \stackrel{(\ref{R12})}{\leq}
\big| \sum_{k=1}^{\infty}\psi_{k}(w')\bar{\partial}\bar{h}_{k}(w') \big|^{2}\nonumber\\
& \stackrel{(\ref{R11})}{\lesssim}
\sum_{k,\psi_{k}(w')\neq 0} \psi_{k}(w')
\bigg\{(\rho^{-1}(w'))^{2}G_{2,\Tilde{R}}(f)(w')^{2}\nonumber\\
&\qquad\qquad\qquad\qquad\qquad+\reallywidehat{|H|^{2}}_{m_{1}m_{2}r}(w')\bigg\} \nonumber\\
&\stackrel{(\ref{R13})}{\lesssim}\big(\rho^{-1}(z))G_{2,s}(f)(z)\big)^{2}+\reallywidehat{|H|^{2}}_{m_{1}r}(z).
\end{align}
Hence as mentioned in (\ref{R7}), and applying Theorem \ref{Thm1.1},
\begin{align}\label{R15}
    \|H_{\bar{f}_{1}}\|^{2}_{S_{2}}&\lesssim \|f\|^{2}_{\operatorname{IDA}^{2,2,-1}_{s}}+ \int_{\mathbb{C}}\big[(\reallywidehat{|H|^{2}}_{m_{1}m_{2}r}(z))^{1/2}\big]^{2}dA(z)\nonumber\\
    &\lesssim \|f\|^{2}_{\operatorname{IDA}^{2,2,-1}_{s}}+
\int_{\mathbb{C}}\big(\rho^{-1}(z))G_{2,s}(f)(z)\big)^{2}dA(z)+ \int_{\mathbb{C}}\reallywidehat{|H|^{2}}_{m_{1}r}(z)dA(z)\nonumber\\
&\lesssim 
\|f\|^{2}_{\operatorname{IDA}^{2,2,-1}_{s}}+ \int_{\mathbb{C}}|H|^{2}dA\nonumber\\
&\lesssim \|f\|^{2}_{\operatorname{IDA}^{2,2,-1}_{s}},
\end{align}
where in the last line we have used (\ref{RR6}).\\
This together with (\ref{B6}) implies that
\begin{equation}
    \|H_{\bar{f}}\|_{S_{2}}\lesssim \|H_{f}\|_{S_{2}}.
\end{equation}
We are done since the proof is symmetric for $f$ and $\bar{f}$.
\end{proof}

We make the following remark related to the Berger-Coburn phenomenon for other values of $p$.

\begin{Rem}\label{Muckenhoupt}
For $1<p<\infty$ we say that $\omega$ is a Muckenhoupt weight and write $\omega\in A_{p}$ if there is a constant $C>0$ such that for any disk $B\subset\mathbb{C}$, we have
\begin{equation}
    \left(\frac{1}{|B|}\int_{B}\omega dA\right)
    \left(\frac{1}{|B|}\int_{B}\omega^{-q/p} dA\right)^{p/q}\leq C<\infty,
\end{equation}
where $q$ is the H\"older conjugate of $p$ and $|B|$ is the Lebesgue measure of $B$. As shown in \cite{dragivcevic2011weighted}, if $\omega\in A_{p}$ and $1<p<\infty$, then the Ahlfors-Beurling operator
\begin{equation}
    \mathcal{I}(f)(z)=p.v.-\frac{1}{\pi}\int_{\mathbb{C}}\frac{f(\xi)}{(\xi-z)^{2}}dA(z)
\end{equation}
is bounded on $L^{p}(\omega)$. Hence, similarly to the proof of Lemma 7.1 in \cite{Janihu2022schatten}, we can show that when $f$ is bounded,
\begin{equation}
    \|\partial f\|_{L^{p}(\omega)} \leq C \|\bar{\partial} f\|_{L^{p}(\omega)},
    \end{equation}
where $C$ is a constant depending only on $p$.

To generalize Theorem~\ref{Thm1.3} to the other values of $1<p<\infty$, our approach would require only one additional ingredient that $\omega=\rho^{p-2}$ is a Muckenhoupt weight (see \eqref{R6}). However, we have not been able to prove this condition and also note that Lemma~\ref{lem-RadiusEquivalence} does not seem to help because the constants $c_r$ in \eqref{P2} are not bounded in general.

\end{Rem}

Next, we consider the case $0<p\le 1$. Recently Xia~\cite{xia2023berger} defined the following simple function
\begin{equation}\label{Xia}  
f(z) := 
     \begin{cases}
       \frac{1}{z} &\quad\text{if } |z|\geq 1,\\
       0 &\quad\text{if } |z|<1.\\ 
     \end{cases}
 \end{equation}
and used it to show that the Berger-Coburn phenomenon does not hold for trace class Hankel operators on the classical Fock space. Hu and Virtanen \cite{hu2023berger} noticed that when $0<p\leq 1$ the same example shows that there is no Berger-Coburn for Schatten class Hankel operators on generalized Fock spaces. Here we use Xia's example again to prove that the Berger-Coburn phenomenon fails for some $S_p(F^2_\phi, L^2_\phi)$ while it remains open whether it fails for the remaining doubling Fock spaces.

\begin{proof}[Proof of Theorem \ref{thm1.4}]
To prove the theorem, we use Theorems \ref{Thm1.2} and \ref{Thm5.4}. The idea is to find a bounded function $f$ with $f\in \operatorname{IDA}^{p,2,-2/p}_{r}$ such that $f\notin \operatorname{IMO}^{p,2,-2/p}_{r}$ for some (equivalent any) $r>0$. Note that by remark 1 in \cite{marco2003interpolating}, there are constants $C,\eta>0$, and $0\leq\beta<1$ such that for $|z|>1$,
\begin{equation}\label{B11}
    C^{-1}|z|^{-\eta}\leq \rho(z)\leq C|z|^{\beta}.
\end{equation}
Let $f$ be as in (\ref{Xia}).
By Theorem \ref{Thm1.1}, the definition of $\operatorname{IDA}^{p,2,-2/p}_{r}$ is independent of $r$. So for simplicity, we set $r=1$. It is easy to see that for a large enough $R>0$, and $|z|\geq R$, $f$ is holomorphic in $D^{1}(z)=D(z,\rho(z))$, and hence trivially $G_{2,1}(f_{\beta})(z)=0$. Indeed, one can see that for $|z|\geq R$, $D^{1}(z)\cap D(0,1)=\emptyset$. Moreover, for all $|z|<R$, there is a constant $C$ such that
\begin{equation}
    G_{2,1}(f)(z)<C,
\end{equation}
as $f$ is bounded in the bounded domain $D^{1}(z)$. Thus,
\begin{align}\label{B14}
\|f\|_{\operatorname{IDA}^{p,2,-2/p}_{1}}^{p}&=\|\rho^{-2/p}G_{2,1}(f)\|_{L^{p}}^{p}
= \int_{\mathbb{C}}\rho^{-2}G_{2,1}(f)^{p}dA\nonumber\\
        &\leq C\int_{|z|<R}\rho^{-2}dA 
        <\infty.
\end{align}
Indeed, by Theorem 14 in \cite{marco2003interpolating}, there is a smooth function $\psi$, where $\Delta\psi dA$ is doubling and $\Delta\psi \simeq \rho_{\psi}^{-2}\simeq \rho^{-2}$. Hence,
\begin{equation}
    \int_{|z|<R}\rho^{-2}dA\simeq 
    \int_{|z|<R}\Delta\psi dA< \infty,
\end{equation}
as the doubling measures are locally finite. So by (\ref{B14}), $f\in\operatorname{IDA}^{p,2,-2/p}_{1}$, and Theorem \ref{Thm1.2} implies that $H_{f}\in S_{p}$.

To show that $H_{\Bar{f}}\notin S_{p}$, note that if $|z|\geq R$, $\Bar{f}$ is harmonic on $D^{1}(z)$ and by the mean-value property of harmonic functions,
\begin{equation}
    \widehat{\Bar{f}_{1}}(z)=\frac{1}{|D^{1}(z)|}\int_{D^{1}(z)}\Bar{f}dA=\bar{f}(z).
\end{equation}
Moreover, by definition, $MO_{2,r}(f)(z)=MO_{2,r}(\Bar{f})(z)$, and thus for $|z|\geq R$,
\begin{align}\label{B17}
        MO_{2,1}(f)(z)&=\left( \frac{1}{|D^{1}(z)|}\int_{D^{1}(z)}|\Bar{f}(w)-\Bar{f}(z)|^{2} dA(w) \right)^{1/2}\nonumber\\
        &=
        \left( \frac{1}{|D^{1}(z)|}\int_{D^{1}(z)}|\frac{1}{\Bar{w}}-\frac{1}{\Bar{z}}|^{2}dA(w)  \right)^{1/2}\nonumber\\
        &=
        \left( \frac{1}{|D^{1}(z)|}\int_{D^{1}(z)}\frac{|w-z|^{2}}{|zw|^{2}}dA(w)  \right)^{1/2}.
\end{align}
For $w\in D^{1}(z)$, we can write $w=z+re^{i\theta}$ where $0\leq r<\rho(z)$ and $0\leq \theta<2\pi$. Therefore,
\begin{align}
    \int_{D^{1}(z)}\frac{|w-z|^{2}}{|zw|^{2}}dA(w) 
    &=\frac{1}{|z|^{2}} \int_{0}^{\rho(z)}r^{3}\int_{0}^{2\pi} \frac{d\theta dr}{|z+re^{i\theta}|^{2}}
\end{align}
Let $z=|z|e^{i\psi}$. Then
\begin{align}
    \int_{0}^{2\pi} \frac{d\theta}{|z+re^{i\theta}|^{2}}&= \int_{0}^{2\pi} \frac{d\theta}{\big||z|+re^{i\theta}\big|^{2}}
    = \int_{0}^{2\pi} \frac{d\theta}{|z|^{2}+r^{2}+2|z|r \cos{\theta}}.
\end{align}
Defining $y=\frac{r}{|z|}$,
\begin{align*}
    \frac{1}{|z|^{2}}\int_{0}^{\rho(z)} \int_{0}^{2\pi}\frac{r^{3}d\theta dr}{|z|^{2}+r^{2}+2|z|r \cos{\theta}}
    &=\frac{1}{|z|^{2}}\int_{0}^{\frac{\rho(z)}{|z|}} \int_{0}^{2\pi}\frac{y^{3}|z|^{4}d\theta  dy}{|z|^{2}+y^{2}|z|^{2}+2|z|^{2}y \cos{\theta}}\\
    &= \int_{0}^{\frac{\rho(z)}{|z|}} \frac{y^{3}}{2y}\int_{0}^{2\pi}\frac{d\theta dy}{\frac{1+y^{2}}{2y}+\cos{\theta}}.
\end{align*}
Let $x=\frac{1+y^{2}}{2y}$. Then
\begin{align}
    \int_{0}^{2\pi}\frac{d\theta}{\frac{1+y^{2}}{2y}+\cos{\theta}}= \int_{0}^{2\pi}\frac{d\theta }{x+\cos{\theta}}.
\end{align}
Taking $t=\tan{\frac{\theta}{2}}$, we have $\theta=2\tan^{-1}(t)$, $d\theta=\frac{2dt}{1+t^{2}}$, and $\cos{\theta}=\frac{1-t^{2}}{1+t^{2}}$.
Since the cosine function is even, one has
\begin{align}
   \int_{0}^{2\pi}\frac{d\theta }{x+\cos{\theta}}&=2 \int_{0}^{\pi}\frac{d\theta }{x+\cos{\theta}}=2\int_{0}^{\infty}\frac{2dt}{x(1+t^{2})+1-t^{2}}\nonumber\\
   &=2\int_{0}^{\infty}\frac{2dt}{t^{2}(x-1)+(x+1)}
= \frac{4}{x+1}\int_{0}^{\infty}\frac{dt}{1+(\frac{x-1}{x+1})t^{2}}.
\end{align}
Taking $u=\sqrt{\frac{x-1}{x+1}}t$, we obtain
\begin{align}
    \frac{4}{x+1}\int_{0}^{\infty}\frac{dt}{1+(\frac{x-1}{x+1})t^{2}}&=\frac{2}{x+1}\int_{0}^{\infty}\frac{2\sqrt{\frac{x+1}{x-1}}du}{u^{2}+1}=
    \frac{2}{x+1}\sqrt{\frac{x+1}{x-1}}\int_{0}^{\infty}\frac{2du}{u^{2}+1}\nonumber\\
    &= \frac{2}{x+1}\sqrt{\frac{x+1}{x-1}} \int_{0}^{\pi} d\theta=
    \frac{2\pi}{\sqrt{(x-1)(x+1)}}\nonumber\\
    &= \frac{2\pi}{\sqrt{(\frac{1+y^{2}}{2y}-1)(\frac{1+y^{2}}{2y}+1)}}= 
    \frac{4\pi y}{(1-y)(1+y)}.
\end{align}
Thus,
\begin{align}
    \int_{0}^{\rho(z)/|z|} \frac{y^{3}}{2y}\int_{0}^{2\pi}\frac{d\theta dy}{\frac{1+y^{2}}{2y}+\cos{\theta}}= \int_{0}^{\rho(z)/|z|}\frac{y^{2}}{2}\frac{4\pi ydy}{(1-y^{2})}.
\end{align}
Let $v= y^{2}$, then 
\begin{align}
    \int_{0}^{\rho(z)/|z|}\frac{y^{2}}{2}\frac{4\pi ydy}{(1-y^{2})}&= 
    \int_{0}^{(\rho(z)/|z|)^{2}}\frac{v}{2}\frac{4\pi\sqrt{v}dv}{(1-v)}\frac{dv}{2\sqrt{v}}\nonumber\\
    &=\pi \int_{0}^{(\rho(z)/|z|)^{2}} \frac{v-1+1}{1-v}dv= \pi \int_{0}^{(\rho(z)/|z|)^{2}} \left(-1+\frac{1}{1-v}\right)dv\nonumber\\
    &=\pi\left[ -(\frac{\rho(z)}{|z|})^{2}-\ln{(1-(\frac{\rho(z)}{|z|})^{2})}
    \right].
\end{align}
Hence,
\begin{align}
    MO_{2,1}(f)(z)= \frac{\pi}{\rho(z)}
    \left[ -(\frac{\rho(z)}{|z|})^{2}-\ln{(1-(\frac{\rho(z)}{|z|})^{2})}
    \right]^{1/2}.
\end{align}
Therefore,
\begin{align}
    \|f\|^{p}_{\operatorname{IMO}^{p,2,-2/p}_{1}}&=\int_{\mathbb{C}}\rho(z)^{-2}\operatorname{MO}_{2,1}(f)(z)^{p}dA(z)\nonumber\\
    &\simeq\int_{\mathbb{C}}\frac{1}{\rho(z)^{2}}\frac{1}{\rho(z)^{p}} \left[ -(\frac{\rho(z)}{|z|})^{2}-\ln{(1-(\frac{\rho(z)}{|z|})^{2})}
    \right]^{p/2}dA(z). 
\end{align}
Note that taking $x=-(\rho(z)/|z|)^{2}$, the term in the bracket is $x-\ln{(1+x)}=x-x+x^{2}/2-x^{3}/3+\cdot\cdot\cdot$, and hence the most contribution comes from the term $x^{2}/2$. Thus,
\begin{align}
    \|f\|^{p}_{\operatorname{IMO}^{p,2,-2/p}_{1}}&\simeq \int_{\mathbb{C}}\frac{1}{\rho(z)^{p+2}}\frac{\rho(z)^{2p}}{|z|^{2p}}dA(z)=\int_{\mathbb{C}}\frac{1}{\rho(z)^{2-p}}\frac{1}{|z|^{2p}}dA(z)\nonumber\\
    &\geq \int_{|z|\geq R}\frac{1}{|z|^{\beta(2-p)}}\frac{1}{|z|^{2p}}dA(z)\simeq \int_{R}^{\infty} \frac{rdr}{r^{2p+\beta(2-p)}}\nonumber\\
    &=\int_{R}^{\infty}r^{1-2p-\beta(2-p)}dr\simeq r^{2-2p-\beta(2-p)}\big|_{r=R}^{\infty}.
\end{align}
Note that $2-2p-\beta(2-p)=(2-p)\big(\frac{2(1-p)}{2-p}-\beta\big)$, and since $0<p\leq 1$, the integral diverges when $\beta\leq \frac{2(1-p)}{2-p}$. So, when $p=1$, $\beta$ must be zero. When $p$ is very close to zero, $\beta$ can get very close to 1, implying that Xia's example is a counterexample for any doubling measure. 

\end{proof}

\begin{Rem}

One could also hope to modify \eqref{Xia} so that it takes into account the growth condition of $\rho$; see \eqref{e:growth condition}. However, there are no holomorphic functions that behave like $|z|^c$ at infinity unless $c$ is an integer. Indeed, suppose that $f$ is holomorphic in the complement of a disk centered at the origin, and assume that $\sup_{\theta} |f(re^{i\theta})|\simeq r^{c}$ as $r\to \infty$. Then $c\in\mathbb{Z}$. To see this, for such a function $f$, set $g(z)=z^{k}f(1/z)$, where $k\leq c$ is an integer. Then $g$ has a removable singularity at the origin since $|g(re^{i\theta})|=\mathcal{O}(r^{k-c})$ as $r\to 0$. So $g$ is bounded at zero, and hence $g$ has a power series $\sum a_{k}z^{k}$ near the origin, which implies that $c\in \mathbb{Z}$.

Finally, notice that substituting \eqref{Xia} in the proof of Theorem~\ref{thm1.4} by the functions $f(z) = 1/z^n$ for $|z|>1$ and $f(z)=0$ elsewhere actually works worse when the integer $n$ is larger than $1$.
\end{Rem}

\begin{proof}[Proof of Corollary~\ref{canonical weight cor}]
We apply Theorem \ref{thm1.4} to the canonical doubling weights $\phi(z)=|z|^{m}$ with $m>0$. Recall that by Lemma \ref{canonical weight lem}, there is some $R>0$ such that $\rho(z)\leq |z|^{1-m/2}$ for $|z|\geq R$. Therefore, $\beta_{\phi}=1-m/2$. We can conclude that the Berger-Coburn phenomenon fails for $S_{p}(F^{2}_{|z|^{m}},F^{2}_{|z|^{m}})$ if $1-m/2\leq \frac{1-p}{1-p/2}$, which is equivalent to $m\geq \frac{p}{1-p/2}$.
In particular, if $m\geq 2$, then the phenomenon fails for all Schatten classes $S_{p}$ with $0<p\leq 1$.
\end{proof}

\printbibliography
\end{document}